\UseRawInputEncoding
\documentclass[reqno]{amsart}
\usepackage{amssymb,amsmath,amsthm,amstext,amsfonts}
\usepackage[dvips]{graphicx}
\usepackage{psfrag}
\usepackage{url}
\usepackage{amsmath,amstext,amsthm,amsfonts}
\usepackage{color}
\usepackage{xcolor}
\usepackage[colorlinks=true, linkcolor=blue, urlcolor=red, citecolor=blue, hyperindex, backref]{hyperref}

\usepackage{tikz}
\usepackage{ifthen}
\usepackage{tikz-3dplot}

\makeatletter \@addtoreset{equation}{section} \makeatother

\renewcommand\thetable{\thesection.\@arabic\c@table}

\theoremstyle{plain}
\newtheorem{maintheorem}{Theorem}

\newtheorem{mainconjecture}{Conjecture}

\newtheorem{maincorollary}{Corollary}

\newtheorem{mainproblem}{Problem}

\newtheorem{mainquestion}{Question}

\newtheorem{theorem}{Theorem }[section]
\newtheorem{proposition}[theorem]{Proposition}
\newtheorem{lemma}[theorem]{Lemma}
\newtheorem{corollary}[theorem]{Corollary}

\theoremstyle{definition} \theoremstyle{remark}
\newtheorem{remark}[theorem]{Remark}
\newtheorem{example}[theorem]{Example}
\newtheorem{definition}[theorem]{Definition}

\newcommand{\Sc}{\mathbb{S}}
\newcommand{\Lo}{\mathcal{L}}

\newcommand{\C}{\mathbb{C}}

\newcommand{\N}{\mathbb{N}}
\newcommand{\R}{\mathbb{R}}

\newcommand{\cE}{\mathcal{E}}

\newcommand{\cS}{\mathcal{S}}

\begin{document}

\title{From thermodynamic and spectral phase transitions to multifractal analysis}

\author{ Thiago Bomfim,  Victor Carneiro and Afonso Fernandes}

\address{Thiago Bomfim, Departamento de Matem\'atica, Universidade Federal da Bahia\\
Av. Ademar de Barros s/n, 40170-110 Salvador, Brazil.}
\email{tbnunes@ufba.br}
\urladdr{https://sites.google.com/site/homepageofthiagobomfim/}

\address{Victor Carneiro, Departamento de Matem\'atica, Universidade Federal da Bahia\\
Av. Ademar de Barros s/n, 40170-110 Salvador, Brazil.}
\email{victor.carneiro93@gmail.com}

\address{Afonso Fernandes, Departamento de Matem\'atica, Universidade Federal da Bahia\\
Av. Ademar de Barros s/n, 40170-110 Salvador, Brazil.}
\email{afonso$_{-}$43@hotmail.com}

\date{\today}

\begin{abstract}
It is known that all uniformly expanding or hyperbolic dynamics have no phase transition with respect to H\"older continuous potentials. In \cite{BC21}, it is  proved that for all transitive $C^{1+\alpha}-$local diffeomorphism $f$ on the circle, that is neither a uniformly expanding map nor invertible, has a unique thermodynamic phase transition with respect to the geometric potential, in other words, the topological pressure function $\mathbb{R} \ni  t \mapsto  P_{top}(f,-t\log|Df|)$ is  analytic except at a point $t_{0} \in (0 , 1]$. Also it is proved spectral phase transitions, in other words,  the transfer operator $\mathcal{L}_{f,-t\log|Df|}$ acting on the space of H\"older continuous functions, has the spectral gap property for all $t<t_0$ and does not have the spectral gap property for all $t\geq t_0$. Our goal is to prove that thermodynamical and spectral phase transition results imply a multifractal analysis for the Birkhoff spectrum. In particular, we exhibit a class of partially hyperbolic endomorphisms that admit thermodynamical and spectral phase transitions with respect  to the geometric potential, and we describe the multifractal analysis of your central Lyapunov spectrum.
\end{abstract}

\subjclass[2010]{82B26, 37D35, 37C30, 37C40 }
\keywords{Phase transition, Thermodynamical formalism, Transfer operator, Multifractal analysis.}

\maketitle
{\tiny
\tableofcontents
}

\section{Introduction}

This paper is concerned with the phase transition problem of smooth dynamical systems and the consequences of this phase transition for multifractal analysis. From a dynamical systems point of view, (thermodynamics) phase transitions are often associated with the lack of differentiability or analyticity of the topological pressure, as a function of the potential.

 More precisely, given $f : \Lambda \rightarrow \Lambda$  a continuous dynamical system on a compact metric space $\Lambda$  and $\phi : \Lambda \rightarrow \R$ a continuous potential, the variational principle for the pressure asserts that 
$$
P_{top}(f, \phi) = \sup\{h_{\mu}(f) + \int\phi d\mu : \mu  \text{ is an } f-\text{invariant probability}\}
$$
where $P_{top}(f, \phi)$ denotes the topological pressure of $f$ with respect to $\phi$ and $h_{\mu}(f)$ denotes the Kolmogorov-Sinai's metric entropy.  An equilibrium state $\mu_{\phi}$ for $f$ with respect to $\phi$ is a probability measure that attains
the supremum. In the special case that $\phi \equiv 0$, the topological pressure $P_{top}(f, \phi)$ coincides with  the topological entropy  $h_{top}(f)$, which is one of the most important topological invariants in dynamical systems (see e.g. \cite{W82}). 
Thereby, we say that $f$ has a phase transition with respect to $\phi$ if the topological pressure function
$$
\R \ni t \mapsto P_{top}(f, t\phi)
$$
is not analytic.

It follows from \cite{S72, Bow75, BR75} that if $f$ is a hyperbolic or expanding dynamics and $\phi$ is H\"older continuous then $f$ doesn't have phase transition with respect to $\phi$. Indeed,  the thermodynamical properties can be recovered through the Ruelle-Perron-Frobenius operator or transfer operator $\mathcal{L}_{\phi}$ acting on functions $g : M \rightarrow \C$ defined as the following:
 $$
 \mathcal{L}_{f,\phi}(g)(x) := \sum_{f(y) = x}e^{\phi(y)}g(y).
 $$ Using the fact that $\mathcal{L}_{\phi}$ has the spectral gap property (see precise definition in Section \ref{Statement of the main results}) acting on a suitable Banach space\footnote{Usually the suitable Banach space is the H\"older continuous spaces, smooth function spaces, distributions spaces, among others.} it is possible to show the non-existence of phase transition (see e.g. \cite{PU10}). Furthermore, we know that the spectral gap property is also helpful to the study of other finer statistical properties of thermodynamical quantities as large deviation and limit theorems, stability of the topological pressure and equilibrium states or differentiability results for thermodynamical quantities (see e.g. \cite{Ba00,GL06, BCV16,BC19}).

  On the other hand, several non-uniformly hyperbolic examples are known to present phase transitions with respect to regular potentials. For example: a large class of maps of the interval with indifferent fixed point, certain non-degenerate smooth interval maps, porcupine horseshoes, between others (see e.g. \cite{Lo93, PS92, CRL13, CRL15, CRL19, CRL21, DGR14, IRV18, V17}.
 Despite that, a characterization of the set of dynamics with phase transitions or the associated transfer operators doesn't have the spectral gap property on a suitable Banach space seems still out of reach.

In \cite{BC21} proposed the following problem and conjecture:

\begin{mainproblem}\label{probA}
What mechanisms are responsible for the existence of thermodynamical phase transitions for $C^{1}-$local diffeomorphism, with positive topological entropy, and H\"older continuous potentials?
\end{mainproblem}

 \begin{mainconjecture}\label{conjA}
Let $f : M\rightarrow M$ be a $C^{2}-$local diffeomorphism on a Riemannian manifold $M$. If $f$ is not a uniformly expanding or uniformly hyperbolic endomorphism then there exists a suitable potential $\phi$ such that $\mathcal{L}_{f, \phi}$ has no the spectral gap property acting on a suitable Banach space (H\"older continuous or smooth functions).
\end{mainconjecture}

As a first step in studying these issues, in \cite{BC21} answers were obtained for the Problem \ref{probA} and Conjecture \ref{conjA} when $M = \Sc^1$ is the circle. More formally:


\begin{remark}
Given $r \geq 1$ an integer and $\alpha \in (0 , 1]$ we denite by $C^{r}(\Sc^1 , \C)$ and $C^{\alpha}(\Sc^1 , \C)$ the Banach spaces of $C^r$ functions and  $\alpha-$H\"older continuous complex functions whose domain is $\Sc^1$, respectively. 
\end{remark}

\begin{theorem}\label{thC}	\cite{BC21}
Let $E=C^{\alpha}(\Sc^{1} , \C)$ or $C^{r}(\Sc^{1} , \C)$ be and let $f:\Sc^1 \to \Sc^1$ be a transitive non invertible $C^{1}-$local diffeomorphism with $Df \in E$. If $f$ is not an expanding dynamics then there exists  $t_{0} \in (0 , 1]$ such that the transfer operator  $\mathcal{L}_{f, -t\log|Df|}$ has the spectral gap property on $E$ for all $t < t_{0}$ and has no the spectral gap property for all $t \geq t_{0}$.
\end{theorem}

As a consequence, of the previous theorem, was obtained an effective thermodynamic phase transition:

\begin{corollary}\label{mainthD}	\cite{BC21}
Let $f:\Sc^1 \to \Sc^1$ be a transitive non invertible $C^{1}-$ local diffeomorphism with $Df$ H\"older continuous. If $f$ is not an expanding dynamics then the topological pressure function $\R \ni t \mapsto P_{top}(f , -t\log |Df|) $ is  analytical, strictly decreasing and strictly convex in $(-\infty , t_{0})$ and constant equal to zero in $[t_{0} , +\infty)$ (see figure \ref{graphic P}). In particular, $f$ has a unique thermodynamic phase transition with respect to geometric potential $-\log|Df|$ in $t_{0}$.
\end{corollary}

\begin{figure}[htb]
\centering
\includegraphics[width=8.1cm, height=4.5cm]{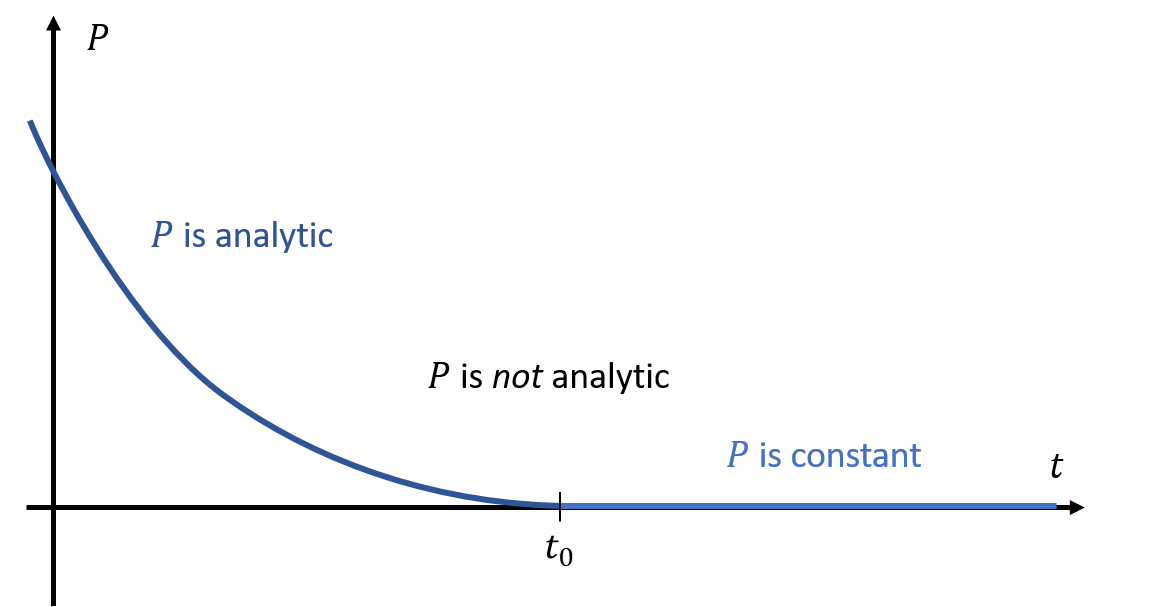}
\caption{Graphic of the topological pressure function}
\label{graphic P}
\end{figure}

We observe that, by \cite{W92}, if $f$ is expansive and has finite topological entropy then the lack of differentiability of the topological pressure function is related to the non-uniqueness of the equilibrium states associated with potential $\phi$. Furthermore, when we have the spectral gap property, the derivatives of the topological pressure function are related to thermodynamics quantities such as equilibrium states and standard deviation (see e.g. \cite{S12}).

It is necessary to understand other thermodynamic consequences for the existence of thermodynamic and spectral phase transitions. As the first step in this direction, we are interested in the implications for multifractal analysis.

In multifractal analysis, we study invariant sets and measures with a multifractal structure. We are essentially measuring the size of those sets, in the sense of Hausdorff dimension or topological entropy, for instance.
Given $\psi_n: M\rightarrow\mathbb{R}$ a sequence of continuous functions and $I \subset \R$ an interval, define
\begin{align*}
X(I)=\Big\{x\in M;\lim\limits_{n\rightarrow\infty}\frac{1}{n}\psi_n(x)\in I\Big\}.
\end{align*} A Multifractal Analysis for this sequence  means understanding these sets metrically, geometrically (Hausdorf dimension), topologically (topological entropy, topological pressure), thermodynamically (variational principles), etc...   

In dynamical systems, given $F : M \rightarrow M$ a continuous transformation on a compact metric space, we are particularly interested in the following sequences
\begin{itemize}
\item $\psi_n=\sum_{i=0}^{n-1}\varphi\circ F^i$ (Birkhoff averages)
\item $\psi_n=\log||DF^n||$ and $\varphi_n=\log||DF^{-n}||^{-1}$ (extremal Lyapunov exponents)
\item $\psi_n=\log\mu(B(.,n,\epsilon)),$ where $\mu$ is an $F-$invariant probability and $B(x,n,\epsilon)$ is a dynamical ball (local entropy)
\end{itemize}

Suppose now that the dynamics $F$ is a codimension one partially hyperbolic endomorphism on 
the manifold $M$, that is, $F$ is a $C^1-$local diffeomorphism and there are a continuous splitting  of the tangent bundle $TM = E^s \oplus E^c \oplus E^u$, positive constants $\lambda $, $\sigma$ and $c$ such that:
\begin{enumerate}
    \item $E^s$ is $DF-$invariant;
    \item $||DF^n_{|E^s} || < e^{-\sigma n + c}$ and $||DF^n_{|E^u} || >^{-\lambda n - c}$;
    \item $||DF^n_{|E^c} ||  > e^{\sigma n - c}||DF^n_{|E^s} || $ and $||DF^n_{|E^c} ||  < e^{-\lambda n + c}||DF^n_{|E^u} || $;
    \item $E^c$ is one dimensional.
\end{enumerate}
The subbundles $E^s, E^c$ and $E^u$ are called stable, central and unstable, respectively. Observe that $E^c$ is $DF-$invariant and the subset of codimensional one partially hyperbolic endomorphism is open in the space of local diffeomorphisms. The partially hyperbolic endomorphisms are a natural generalization of the partially hyperbolic diffeomorphism for the non-invertible context. When $M = \mathbb{T}^2$ is the two dimensional torus and $E^s = \{0\}$, we know that there exists an open subset of robust transitive partially hyperbolic endomorphisms that are not expanding dynamics \cite{LP13}; furthermore a generic partially hyperbolic endomorphism admits finitely many ergodic physical measures and the union of its basins of attraction has total Lebesgue measure \cite{T05}.

Since  $E^c$ is $DF-$invariant,  then the central Lyapunov exponent 
$$\lambda^c(x) = \lim\frac{1}{n}\log ||DF^n_{|E^c_x}||$$ is well defined on a total probability subset (see e.g. the Oseledets theorem in \cite{V14}). We are interested in the Multifractal Analysis of $\lambda^c$, that is, of the sequence $\psi_n(x):= \log ||DF^n_{|E^c_x}||$. Since $E^c$ is one dimensional, then $\psi_n = \sum_{j = 0}^{n-1}\varphi \circ F^{j}$, where $\varphi := \log ||DF_{|E^c}||$, which connects the central Lyapunov exponent with the Birkhoff's theorem.

The study of the topological pressure or Hausdorff dimension of the subsets where the Birkhoff average diverges or converges to a fixed interval can be traced to Besicovitch, and this topic has had contributions by many authors in recent years
(see \cite{DK, Climenhaga, GR, JR,PW97,PW99,T,TV99, Daniel, ZC13,BV17,IJT17,JR21}).

Therefore, given $[a , b] \subset \R$ a closed interval, we define the following sets:
$$
L^{c}_{a , b} := X_{\varphi}([a , b])=\{x\in M; \lambda^c(x)\in [a , b]\} \text{ and }$$
$$
\hat{L} :=  \{ x \in M : \nexists \lim \frac{1}{n}\log ||DF^{n}(x)_{|E^c_x}||\}.
$$
In particular, $L^{c}(\alpha):= L^{c}_{\alpha,\alpha}$ is the set of points $x\in M$ which the central Lyapunov central is exactly $\alpha$. We notice that the previous sets decompose the space of orbits.
From the ergodic point of view, given $\mu$ an ergodic $F$-invariant probability we have $\mu(L^{c}_{[a,b]})$ is $0$ or $1$ by the Birkhoff ergodic theorem, depending on whether $\int\varphi d\mu \notin I$ or $\int\varphi d\mu\in [a,b]$, respectively. Moreover $\mu(\hat{L}) = 0$. So these sets may not be that interesting from the measure point of view, but we have to be careful because these sets can be topologically large. By Thompsom \cite{T} and Lima-Varandas \cite{LV21}, if $F$ has topological properties such as the \emph{specification property} or \emph{gluing orbit property} and $\hat{L} \neq\emptyset$, then $\hat{L}$ is a Baire residual subset and $P_{\hat{L}}(f,\phi)=P_{top}(f,\phi)$, where $P_{Z}(f,\phi)$ is the topological pressure for the continuous potential $\phi$ on the set $Z$.
For this reason, we also consider the sets
$$
E_{a,b} := \hat{L} \cap \Big\{x \in M :   \limsup_{n\to\infty} \frac{1}{n}\log ||DF^{n}_{|E^{c}_{x}}|| \in  [a,b] \text{ or } \liminf_{n\to\infty} \frac{1}{n}\log ||DF^{n}_{|E^{c}_{x}}|| \in  [a,b]\Big \}
$$
$$
\text{and } D_{\mu, \epsilon} := \big\{x \in M : \Big|\lambda^{c}(x) - \int \varphi d\mu \Big| \geq \epsilon\big\}.
$$

We propose then:

\begin{mainquestion}\label{queA}
Given $C_{a, b} = L^{c}_{a, b}$ or $E_{a,b}$, describe the  following functions
$$
 [a,b]\mapsto h_{C_{a,b}}(F) \text{ , }  [a,b]\mapsto HD(C_{a,b}),
 $$
 $$
 \epsilon \mapsto h_{D_{\mu, \epsilon}}(F) \text{ and }\epsilon \mapsto HD(D_{\mu, \epsilon}),
$$
where $h_{Z}(F)$ denotes the topological entropy restricted to the set $Z$ and $HD(Z)$ denotes the Hausdorff dimension of the set $Z$. Do these functions vary continuously? Smoothly? Analytically?
\end{mainquestion}

In \cite{BV17} a connection was established between multifractal analysis, when we measure these sets through topological pressure, and large deviation rates.
The Large Deviations Theory studies, among other things, the rate of convergence at which the temporal average of a sequence of random variables converges to limit distribution. In Dynamical Systems, these ideas help estimate the speed with which the time averages of typical points for ergodic invariant measures converge to the respective space average. In these terms, we are interested in estimating the speed at which
$$
\mu \Big(\Big\{x \in M : \Big|\frac{1}{n}\sum_{i=0}^{n-1}\psi(F^{i}(x)) - \int \psi d\mu \Big| > \epsilon \Big\} \Big)
$$
converges to zero when $n \mapsto +\infty$, where $\mu$ is a $F-$invariant and ergodic probability. (see e.g. \cite{Y90,K90, KN91})

There are some ways to obtain large deviations principles in dynamical systems, one of the ways is through strong properties of the associated transfer operator. With spectral gap in a suitable Banach space, we apply the probabilistic/functional approach derived from theorems Gartner-Ellis's theorem (see e.g. \cite{BCV16,BC19}).
This approximation also provides a lot of information about the so-called free energy function and the large deviation rate function.

In view of what we discussed above, we propose that results  such as spectral phase transition
(Theorem \ref{thC}) should be enough to answer  Question \ref{queA}.
Extending \cite{BC21}, we will show a new class of high-dimensional local diffeomorphisms with thermodynamic and spectral phase transitions. Using this class of local diffeomorphism as a toy model, we will show consequences for the multifractal analysis of central Lyapunov exponents.

This paper is organised as follows. In Section~\ref{Statement of the main results} we provide some definitions
and the statement of the main results on thermodynamical and spectral phase transition, and multifractal analysis. 
In Section~\ref{prelim} we recall the necessary framework on Topological Dynamics, Thermodynamical Formalism and Transfer operator. Sections~\ref{Proofs} and \ref{ProofsB} are devoted to the proof of the main results. 
\section{Statement of the main results}\label{Statement of the main results}

This section is devoted to the statement of the main results.

Throughout the paper we shall denote the circle $\{z \in \C : |z| = 1\}$ by $\Sc^1$ and the $n-$torus $ \Sc^1 \times \cdots \times \Sc^1$ by $\mathbb{T}^n$.

\subsection{Toy model}

This section  introduces the class of local diffeomorphisms to that we will apply our results.

Let $\mathcal{D}^{r}$ be the space of $C^{r}-$local diffeomorphism $F : \mathbb{T}^{d} \times \Sc^1 \rightarrow \mathbb{T}^{d} \times \Sc^1$ such that $F(x , y) = (g(x) , f_{x}(y))$, $\forall (x , y) \in \mathbb{T}^{d} \times \Sc^1$, where:

\begin{enumerate}
    \item $g$ is an expanding linear endomorphism;
    \item $deg(F) > deg(g),$ where $deg$ is the topological degree of a local homeomorphism;
    \item there exists $\eta$, a probability on $\Sc^1$ that is $f_{x}-$invariant, with zero mean Lyapunov exponent $\chi_{\eta}(f_{x}) = \int \log |Df_{x}| d\eta = 0$ for all $x \in \Sc^{1}$;
    \item $F$ is topologically conjugated to a uniformly expanding dynamic.
\end{enumerate}

 \ 
 
Unless we suppose a big expansion of $g$, our toy model will admit dominated decomposition. In this case, $F$ will be a codimensional one partially hyperbolic endomorphism contained in \cite{T05}.

Note that item (2) is equivalent to the existence of an $x \in \Sc^{1}$  with $f_{x}$ not invertible.

It follows from the item (4) and the connectedness of $\mathbb{T}^{d} \times \Sc^1$, that $F$ has the following topological properties: strong transitivity,  periodic specification property and expansivity. Additionally, we have $h_{top}(F) = \log deg(F)$.

Despite our toy model $F$ being in a topological class of an expanding dynamic, it admits zero Lyapunov exponents. In particular, $F$ is not a uniformly expanding dynamic.

Observe that if a local diffeomorphism $F : \mathbb{T}^{d} \rightarrow \mathbb{T}^{d}$ is (positively) expansive then there exists a metric on $\mathbb{T}^{d}$, compatible with the topology, such that the mapping $F$ is a uniformly expanding dynamics with respect to this metric. In particular, $F$ will be topologically conjugated to a uniformly expanding dynamic (for more details, see \cite{CR80}).

The following example shows our context extending the one-dimensional case.

\begin{example}
Let $f: \Sc^1 \rightarrow \Sc^1$ be a transitive non invertible $C^{r}-$local diffeomorphism. It follows from \cite{CM86} that $f$ is topologically conjugated to uniformly expanding dynamics.  Suppose that $f$ is not a uniformly expanding dynamics, then it follows of \cite{CLR03} that $f$ admits an ergodic $f-$invariant probability $\eta$ such that $\chi_{\eta}(f) = \int \log |Df_{x}| d\eta = 0$. Therefore, the skew-product 
$$F : \Sc^1 \times \Sc^1 \xrightarrow[(x , y) \mapsto (2x , f(y))]{} \Sc^1 \times \Sc^1$$ will belong to $\mathcal{D}^{r}$.
\end{example}


The next class of examples will show that our toy model contains an interesting class of local diffeomorphisms in higher dimension; such examples can be seen as differentiable versions of random choices of intermittent dynamics.

Let $x_{1} < x_{2} < \ldots < x_{k} = x_{1}$ be points in $\Sc^1$. Define 
$V(\{x_{j}\}_{j = 1}^{k})$ as the set of  maps $f : \Sc^1 \rightarrow \Sc^1 \text{ transitive } C^{r}-\text{local diffeomorphism such that }
f_{|(x_{j} , x_{j+1})
} \text{ is injective, } 
$
$f([x_{j} , x_{j+1}]) = \Sc^1 , f(x_{j}) = x_{1}, Df(x_{j}) = 1 \text{ and } |Df(x)| > 1
\text{ for all } x \neq x_{j} \text{ and } j = 1 , \ldots, k .$

\begin{proposition}
Fix $x_{1} < x_{2} < \ldots < x_{k} = x_{1}$  points in $\Sc^1$. Let $$F : \mathbb{T}^{d} \times \Sc^1 \xrightarrow[(x , y) \mapsto  (g(x) , f_{x}(y)) ]{} \mathbb{T}^{d} \times \Sc^1$$ be  a $C^{r}-$skew-product such that $g$ is an expanding linear endomorphism and each $f_{x} \in V(\{x_{j}\}_{j = 1}^{k})$. Then $F \in\mathcal{D}^{r}$
\end{proposition}
\begin{proof}
We observe that $\eta := \delta_{x_{1}}$ satisfies item (3) of the definition of $\mathcal{D}^{r}$. To conclude the proposition, it is enough to check the item (4) of the definition. On the other hand, we already know that it is enough to check that $F$ is expansive.

By definition of $V(\{x_{j}\}_{j = 1}^{k})$ and continuity of $x \mapsto f_{x}$, there exists $\epsilon > 0$ such that: given $a < \epsilon$ there is $\lambda_{a} > 1$ with $diam(f_{x}(I)) \geq \lambda_{a} diam(I)$, for all $a 
\leq diam(I) \leq \epsilon$ and $x \in \Sc^1$, where $diam$ means the diameter of a subset
in $\Sc^1$. Suppose by contradiction that $F$ is not expansive. As $g$ is expansive, then there exists $x \in \mathbb{T}^{d}$ and $y_{1} \neq y_{2}$ with $$|f_{g^{n}(x)}
 \circ \cdots \circ f_{x}(y_{1}) - f_{g^{n}(x)}
 \circ \cdots \circ f_{x}(y_{2})| < \epsilon$$
 for all $n \geq 0$. Therefore 
 $$diam(f_{g^{n}(x)}
 \circ \cdots \circ f_{x}([y_{1} , y_{2}]) \geq \lambda_{a}^{n}|y_{1} - y_{2}|,$$ for all $n \leq 0$. This implies that exist $n \geq 0$ such that $$|f_{g^{n}(x)}
 \circ \cdots \circ f_{x}(y_{1}) - f_{g^{n}(x)}
 \circ \cdots \circ f_{x}(y_{2})| > \epsilon,$$ which is absurd. We conclude that $F$ is expansive.
\end{proof}

\begin{example}
For each $\alpha \in [0 , 1]$ define the constants $$
b=b(\alpha):=\left((\frac{1}{2})^{3+\alpha}-\dfrac{4+\alpha}{4+2\alpha}\Big(\frac{1}{2}\Big)^{2+\alpha}\right)^{-1}, \,\, a=a(\alpha):=\dfrac{-b(4+\alpha)}{4+2\alpha}$$ and the polynomial $g_{\alpha} : [0 , \frac{1}{2}] \rightarrow [0 , 1]$ given by $g_{\alpha}(y)=y+ay^{3+\alpha}+by^{4+\alpha}$. We can then define a family of intermittent maps on the circle: 
\begin{equation}\label{difeo}
 f_{\alpha}(y)=\begin{cases}
 g_{\alpha}(y), \text{ if } 0\leq y \leq 1/2 \\
 1-g_{\alpha}(1-y), \text{ if } 1/2 < y \leq 1.
 \end{cases}
\end{equation}

\begin{figure}[h!]
    \centering
    \includegraphics[scale=5]{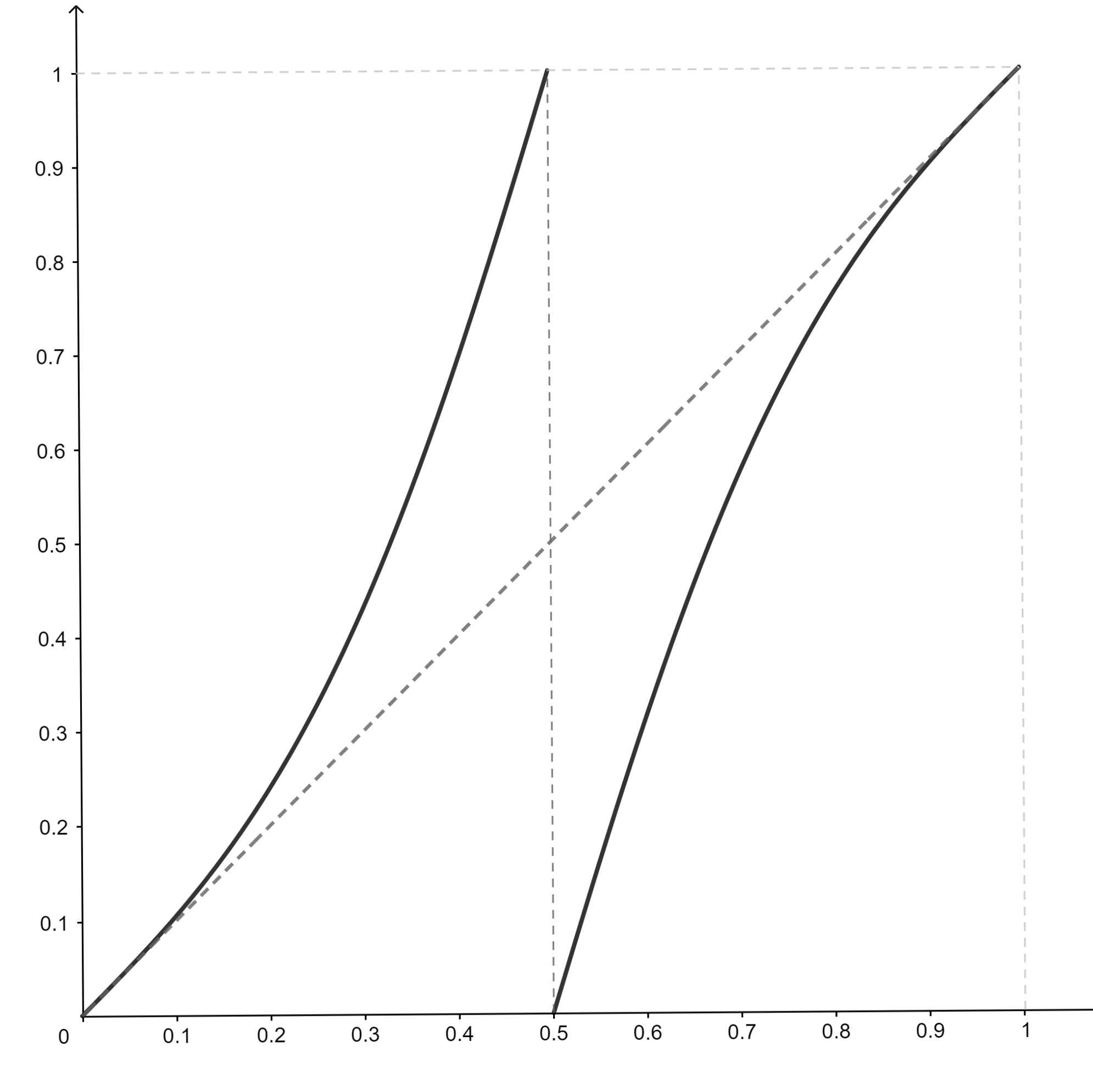}
    \caption{Intermittent map of class $C^2$}
    \label{ManPO.c2}
\end{figure}

This family of dynamics can be seen as a $C^{2}$ version of the Manneville-Pomeau map. Observe that $y = 0$ is an indifferent fixed point and $Df_{\alpha}(y) > 1$ for all $y \neq 0$. In fact, $f_{\alpha} \in V(\{0 , \frac{1}{2}, 1\})$ for all $\alpha \in [0 , 1]$. Therefore, as a byproduct of the previous construction the skew product 
$$F : \mathbb{T}^{d} \times \Sc^1 \xrightarrow[(x , y) \mapsto  (g(x) , f_{x}(y)) ]{} \mathbb{T}^{d} \times \Sc^1,$$ where $g$ is an expanding linear endomorphism, belong to $\mathcal{D}^{2}$.
\end{example}

\subsection{Phase transition}

Our first result ensures effective thermodynamical and spectral phase transition for our toy model. In fact, the phase transition occurs with respect to a geometric potential.
Given $F \in \mathcal{D}^{r}$ we define the potential $\phi^{c}(x , y) := -  \log||\frac{\partial F}{\partial y}(x , y)|| = - \log|Df_{x}(y)|$. Observe that $\phi^{c} = -\log ||DF_{|E^{c}}||$, $\lambda^{c}(z) = -\lim \sum_{i=0}^{n-1}\phi^{c}( F^{i}(z))$ and by Margulis-Ruelle inequality \cite{Rue78} $h_{\mu}\leq h_{top}(g) + \max\{0 ,  \int \log ||DF_{|E^{c}}|| d\mu\}$, for all $F-$invariant probability $\mu$.

\begin{remark}
Given $E$ a complex Banach space and $T : E \rightarrow E$ a bounded linear operator, we say that $T$ has the \emph{(strong) spectral gap property} if there exists a decomposition of its spectrum $sp(T) \subset \C$ as follows: 
$
sp(T) = \{\lambda_{1}\} \cup \Sigma_{1}$ where $\lambda_{1} >0$ is a leading eigenvalue for $T$ with one-dimensional associated eigenspace and there exists $0 < \lambda_{0} < \lambda_{1}$ such that $\Sigma_{1} \subset \{z \in \C : |z| < \lambda_{0}\}$.
\end{remark}

\begin{maintheorem}\label{mainthA}
Let $F \in \mathcal{D}^{r},$ with $r \geq 2$, there exists $t_{0} \in (0 , 1]$ such that:\\

(i) the transfer operator $\mathcal{L}_{F, t\phi^{c}}$ has the spectral gap property on $C^{r-1}$ for all $t < t_{0}$ and has no the spectral gap property for all $t \geq t_{0}$.\\

(ii) the topological pressure function $\R \ni t \mapsto P_{top}(F , t\phi^{c}) $ is analytical, strictly decreasing and strictly convex in $(-\infty , t_{0})$ and constant equal to $h_{top}(g)$ in $[t_{0} , +\infty)$ (see figure \ref{graphic P}).
\end{maintheorem}

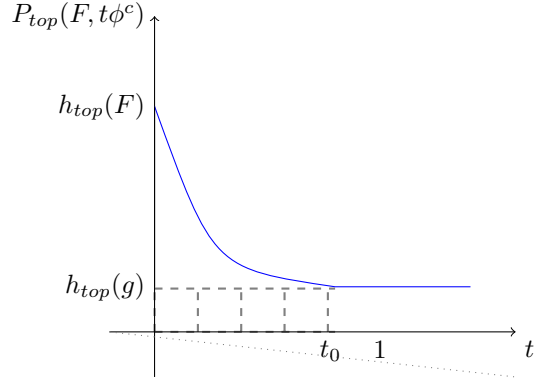
\begin{figure}\label{graphic P}
\begin{tikzpicture}[scale=0.6,domain=0:3]
\draw[thick,color=gray,step= .96 cm,dashed] (0,1) grid (4,0);
\draw[->] (-1,0) -- (8,0) node[below right] {$t$};
\draw[->] (0,-1) -- (0,7) node[left] {$P_{top}(F , t\phi^{c})$};

 \draw[dotted,gray] (-1,0) -- (8,-1);
  \draw[blue] (0,5) .. controls (1.3,1.4) .. (4,1) -- (7,1);

\draw (0,5) node[left] {$h_{top}(F)$} -- (0,1) node[left] {$h_{top}(g)$};
\draw (3.9,0) node[below] {$t_0$};
\draw (5,0) node[below] {$1$};
\end{tikzpicture}
\caption{Graph of the topological pressure function}
\end{figure}

The proof of the previous theorem is analogous to the respective results of \cite{BC21}, with the additional challenge of being in high dimension.

\subsection{Multifractal Analysis}

Given $F \in \mathcal{D}^{r}$, as $F$ is topologically conjugated to a uniformly expanding dynamic and $\mathbb{T}^{d}\times \Sc^{1}$ is connected then $F$ has the periodic specification property. In particular, given a continuous observable $\psi : \mathbb{T}^{d}\times \Sc^{1} \rightarrow \R$ its Birkhoff spectrum $$S_{\psi} := \{\alpha \in \R : \exists x \in \mathbb{T}^{d}\times \Sc^{1} \text{ with }
\lim_{n \to \infty}\frac{1}{n}\sum_{i = 0}^{n-1}\psi(F^{i}(x)) = \alpha\}$$ is a non-empty bounded interval; furthermore $$S_{\psi} = \Big\{\int \psi d\mu : \mu \text{ is an } F- \text{invariant probability}\Big\}$$ (for more details see e.g. \cite{T09}).

In our context, we are interested in studying the fractal sets created from central Lyapunov exponents, that is, taking the observable $\psi = \log \big|\frac{\partial F}{\partial y}\big|$,
$$
L^{c}_{a , b} =\{w\in \mathbb{T}^{d} \times \Sc^{1}; \lim\frac{1}{n}\log \Big|\frac{\partial F^{n}(w)}{\partial y}\Big| \in [a , b]\} \text{ and }$$
$$
E_{a,b} := \Big\{w\in \mathbb{T}^{d} \times \Sc^{1} :  \nexists \lim\frac{1}{n}\log \Big|\frac{\partial F^{n}(w)}{\partial y}\Big| \text{ and } 
$$
$$\limsup_{n\to\infty} \frac{1}{n}\log \Big|\frac{\partial F^{n}(w)}{\partial y}\Big| \in  [a,b] \text{ or } \liminf_{n\to\infty} \frac{1}{n}\log \Big|\frac{\partial F^{n}(w)}{\partial y}\Big| \in  [a,b]\Big \}
$$

 In that case, the Birkhoff spectra of $\psi$ is the set of all central Lyapunov exponents, the central Lyapunov spectrum denoted by $L^{c}(F)$. So it makes sense to consider only intervals contained within the Birkhoff spectra.
 
 Thus denote $\Delta:=\{(a,b) \in L^{c}(F) \times L^{c}(F) : a \leq b \}$, denote $\lambda^{c}_{min}$ and $\lambda^{c}_{max}$ as the infimum and supremum for the central Lyapunov exponents of the equilibrium states obtained from spectral gap (see Section \ref{secLDP} for an appropriate definition), and $\lambda^{c}_{\mu_{0}}$ the central Lyapunov exponent of the measure of maximum entropy of $F$.

  With this notation, we have the following results for the entropy spectra for the central Lyapunov exponents:

\begin{maintheorem}\label{mainthC}
The entropy function $\Delta \ni (a,b) \mapsto h_{top}(L^{c}_{a,b}) = h_{top}(E_{a,b})$ is a concave $C^1$ function satisfying the following:

\begin{itemize}
    \item It is constant and equal to its maximum value $h_{top}(f)$ for $(a,b)$ on the rectangle $[0,\lambda^{c}_{\mu_{0}}]\times[\lambda^{c}_{\mu_{0}},\lambda^{c}_{max}]$;
    \item It is strictly concave and analytic everywhere, except maybe for $b \leq \lambda^{c}_{min}$ in case the exponent $\lambda^{c}_{min}>0$, where the function is linear.
\end{itemize}
\end{maintheorem}

The figure below gives the expected shape for the function $(a,b)\mapsto h_{{L^{c}_{a,b}}(F)}$, for the case where $\lambda_{min}=0$, for simplicity.

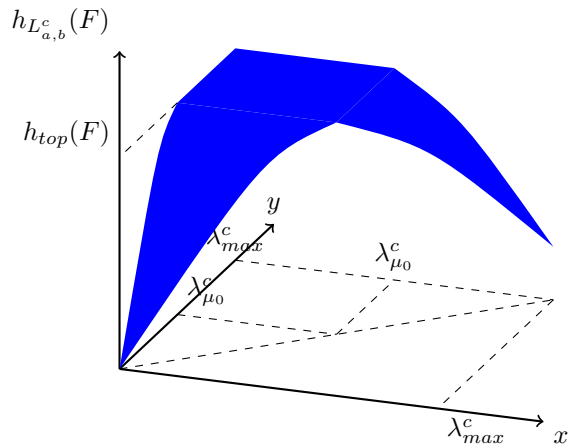
\begin{figure}

\tdplotsetmaincoords{70}{110}
\begin{tikzpicture}[scale=1.5,tdplot_main_coords]
\draw[thick,->] (0,0,0) -- (-4,0,0)  node[anchor=south ]{$y$};
\draw[thick,->] (0,0,0) -- (0,4,0) node[anchor=north west]{$x$};
\draw[thick,->] (0,0,0) -- (0,0,3) node[anchor=south east]{$h_{L_{a,b}^{c}}(F)$};
\tdplotsetcoord{P}{.8}{50}{70}
\draw[dashed,color=black] (0,0,0) -- (-3,3,0) ;
\draw[dashed, color=black] (-3,0,0) node[anchor=south ]{$\lambda_{max}^c$} -- (-3,3,0);
\draw[dashed, color=black] (-1.5,0,0)node[anchor=south west]{$\lambda_{\mu_0}^c$} -- (-1.5,1.5,0);
\draw[dashed,color=black] (-1.5,1.5,0) -- (-3,1.5,0)node[anchor=south]{$\lambda_{\mu_0}^c$};

\draw[dashed, color=black] (-3,3,0) -- (0,3,0) node[anchor=north west]{$\lambda_{max}^{c}$};
\fill[dashed, color=blue] (-1.5,0,2) -- (-1.5,1.5,2) -- (-3,1.5,2) -- (-3,0,2);
\fill[dashed, color=blue] (0,0,0) .. controls (-1,0,1.8) .. (-1.5,0,2) -- (-1.5,1.5,2) .. controls (-1,1,1.8) .. (0,0,0);
\fill[dashed,color=blue]  (-3,3,0.5).. controls (-1.2,2.5,2).. (-1.5,1.5,2) -- (-3,1.5,2) .. controls (-2,2.5,2) .. (-3,3,0.5);

\draw[dashed,color=black] (-1.5,0,2) -- (0,0,2)node[anchor=south east]{$h_{top}(F)$};

\end{tikzpicture}

\caption{Entropy of level sets $L_{a,b}^{c}$}
\end{figure}

The previous theorem is a consequence of the thermodynamic and spectral phase transitions obtained in the Theorem \ref{mainthA}. Therefore analogous theorem will be satisfied for local diffeomorphisms in the circle, taking into account the respective thermodynamical and spectral phase transitions proven in \cite{BC21}. 
Thus we improve the description multifractal of the entropy and Hausdorff spectra for the Lyapunov exponents obtained by Hofbauer \cite{Ho10}.

\begin{maincorollary}\label{mainthD}
Let $f:\Sc^1 \to \Sc^1$ be a transitive non invertible $C^{1}-$local diffeomorphism with $Df$ Holder continuous. Then:

(i) The entropy function $\Delta \ni (a,b) \mapsto h_{L_{a,b}}(f)$ is a concave $C^1$ function satisfying the following:

\begin{itemize}
    \item It is constant and equal to its maximum value $h_{top}(f)$ for $(a,b)$ on the rectangle $[0,\lambda_{\mu_{0}}]\times[\lambda_{\mu_{0}},\lambda_{\max}]$;
    \item It is strictly concave and analytic everywhere, except maybe for $b \leq \lambda_{min}$, in case the exponent $\lambda_{min}>0$, where the function is linear.
\end{itemize}

 (ii) The Hausdorff dimension function $\Delta \ni (a,b) \mapsto HD(L_{a,b})$ is a concave $C^1$ function, which does not depend on $b$, satisfying:
\begin{itemize}
    \item It is constant and equal to its maximum value $t_0$ for $a\leq \lambda_{min}$;
    \item It is strictly concave, decreasing and analytic for $a > \lambda_{min}$.
\end{itemize}
   Note that if $\lambda_{min}>0$ this function is constant on $[0,\lambda_{min})$, and non-analytic for $a=\lambda_{min}$, whereas if $\lambda_{min}=0$ this function is analytic.
\end{maincorollary}

The figures below give the expected shape for the function $(a,b) \mapsto HD(L_{a,b})$ in both cases.

\begin{figure}[h]%
    \centering
    {{\includegraphics[width=7cm]{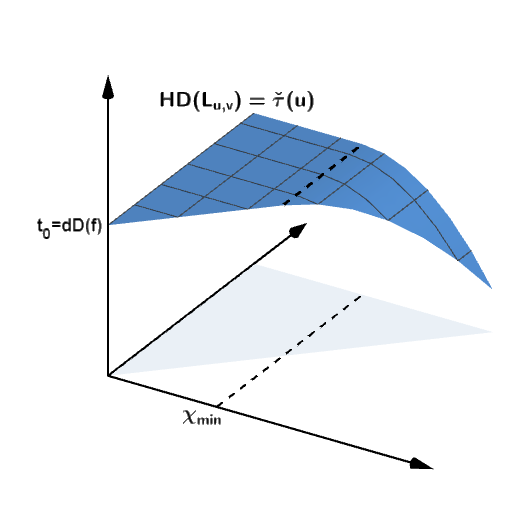} }}%
    \end{figure}
    \begin{figure}[h]%
    \centering
    {{\includegraphics[width=7cm]{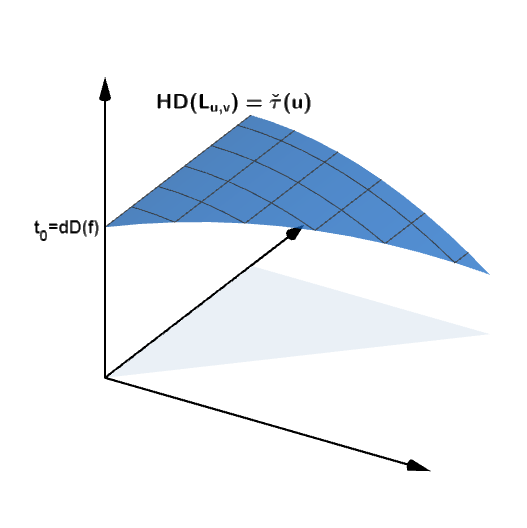} }}%
    \caption{Hausdorff dimension of level sets}%
    \label{fig:example2}%
    
\end{figure}

The proof of item (ii) of the previous Corollary also uses the multifractal description  obtained by Hofbauer \cite{Ho10}. It's possible to give a shorter proof of item (i) of the previous corollaries using the work of Hofbauer directly \cite{Ho10}.

For the circle case, we try to understand under what circumstances $\lambda_{min}>0$ or $\lambda_{min}=0$. For maps which are expanding except on a finite number of \textit{indifferent fixed points} this is determined by the asymptotic expansion near the fixed points: if the map $f$ has strictly $C^{1+p}$ expansion, then the pressure is not differentiable at $t_0=1$ and $\lambda_{min}>0$; if the map $f$ has a $C^2$ expansion, then the pressure is differentiable at $t_0=1$ and $\lambda_{min}=0$ (see Section \ref{difparam}).


\section{Preliminary}\label{prelim}

In this section, we provide some definitions and preliminary results needed to prove the main results. 

\subsection{Topological Dynamics}

We state some classical definitions of topological dynamics that will be useful to us (for more details, see e.g. \cite{OV16}).

Regarding the distances of orbits one concept is essential.

\begin{definition}
The dynamic $f: M \rightarrow M$ is called \textbf{uniformly expanding} if $f$ is an open map and there are constants $\sigma > 1$, $r>0$ and $n \geq 0$ such that 
$$d(f^{n}(x),f^{n}(y))\geq \sigma d(x,y) \text{ for all } d(x,y) < r. $$
\end{definition}

If $M$ is a Riemannian manifold and $f:M \to M$ is a $C^1$ map then $f$ is uniformly expanding if, only if, there exist constants $C,\lambda > 0$ such that $$\|Df^n_x(v)\| \geq Ce^{\lambda n}\|v\|$$
for all $x \in M$, all $v \in T_x M$ and all $n\geq 1$.

If $f: M \rightarrow M$ is an expanding dynamics defined on a connected domain then it is topologically exact, that is, for every open set $A \subset M$ exists $n \in \N$ such that $f^{n}(A) = M$. In particular we have that every set of pre-orbits $\bigcup_{n\geq 0}f^{-n}\{x\}$ is dense on $M$ for every $x\in X$, this property is called of strong transitivity and will be important later.

 Remember that two maps $f: X \to X$ and $g: Y \to Y$ are \textbf{conjugate} if there exists a homeomorphism $h:X\to Y$ such that $g\circ h=h \circ f$.
Topological exactness and dense pre-orbits are in fact \textbf{topologically invariant}, that is, if a map is conjugate to another map with either of these properties then it also does. However, a conjugate of an expanding map need not be expanding as well. 

\

\subsection{Ergodic theory}
Now we state some classical definitions,
 notations and results from Ergodic Theory (for more details see e.g. \cite{OV16}).
 
Given a dynamics $f : M \rightarrow M$, on a compact metric space, we will denote the space of $f-$invariants probabilities by $\mathcal{M}_{1}(f)$ and  the space of $f-$invariants probabilities that are $f-$ergodic by $\mathcal{M}_{e}(f)$.

\subsubsection{Rokhlin formula}

One important property that our dynamics have is admitting a \textbf{generating partition}. Recall that a partition is said to be generating if its pre-images generate the Borelian $\sigma$-algebra. If its domain is a metric space, then a partition such that the diameter of the elements of $\bigvee_{i=m}^{+\infty} f^{-i}(\mathcal{P})$ gets arbitrarily small is a generating partition.

Next, we recall the definition of \textit{Jacobian}, which together with generating partition gives the Rokhlin formula. Let $f : M \rightarrow M$ be a locally invertible map and a given probability $\nu$ (not necessarily invariant), define the \textbf{Jacobian} of $f$ with respect to $\nu$ as the measurable function $J_\nu(f)$, which is essentially unique, satisfying:
$$\nu(f(A))=\int_A J_{\nu}(f) d\nu$$
for any invertibility domain $A$.

The existence of generating partition gives a fundamental tool for calculating the metric entropy:

\begin{theorem}[\bf Rokhlin formula] Let $f: M \to M$ be a locally invertible transformation and $\nu$ be an $f$-invariant probability. Assume that there is some generating, up to measure zero, partition $\mathcal{P}$ such that every $P \in \mathcal{P}$ is an invertibility domain of $f$. Then

$$h_\mu(f)=\int \log J_\mu(f) d\mu.$$

\end{theorem}

\

\subsubsection{Lyapunov exponents}

The Lyapunov exponents translate the asymptotic rates of expansion and contraction of a smooth dynamical system. In a broader context, these are defined via the Oseledets multiplicative ergodic theorem:

 We say that $\lambda$ is a \textit{Lyapunov exponent} for $f : M \rightarrow M$, a $C^1$ map, if there exists a point $x$ and a vector  $v \in T_xM$ such that
$$ \lambda=\lim_{n\to\infty} \dfrac{1}{n} \log \|Df^n_x(v)\|,$$

We let $L(f)$ denote the set of all Lyapunov exponents for $f$.
The Oseledets Ergodic Theorem states that for each ergodic measure $\mu \in M(f)$
there exists constants $\lambda_{1}(f, \mu) > \dots > \lambda_{k}(f,\mu)$, and a filtration of $Df-$invariants subspaces  $T_xM = V_{1}(x) \supset \ldots \supset V_{k+1}(x) = \{0\}$,
such that 
$$\lim_{n \to \infty} \dfrac{1}{n} \log \|Df^n(x) v\|= \lambda_{i}(f , \mu) = \lambda_i$$
for $\mu$-almost every $x$ and every
vector $v \in V_{i}(x) \setminus V_{i+1}(x)$, $i = 1 \ldots, k$. For non-ergodic measures the number $k$, the constants $\lambda_j$, and the tangent bundle decomposition which depends on $x$ may depend on the ergodic component. The constants $\lambda_
j$ are called the Lyapunov exponents associated to  measure $\mu$.

Now for an ergodic measure $\mu$, let $m_i:=m^i_x=\dim E^i_x-\dim E^{i+1}_x$ be the multiplicity for $i=1,\dots,k$   and $\mu-$almost every $x$. Define the sum of positive Lyapunov exponents, considering the multiplicity

$$\lambda_+(\mu)=\sum_{\lambda_j>0} m_j \lambda_j.$$

We will need the following well-known relation between entropy and positive Lyapunov exponents:

\begin{theorem}[Margulis-Ruelle inequality, \cite{Rue78}]
Let $f : M \rightarrow M$ be a $C^{1}$-local diffeomorphism that preserves an $f-$invariant and ergodic probability $\mu$. Then $$h_\mu(f)\leq  \lambda_+(\mu)$$
\end{theorem}

However, for our context: given $F \in \mathcal{D}^{r}$ and an $f-$invariant and ergodic probability $\mu$ we have that its Lyapunov exponents are $\lambda^{c},  \lambda_{1}, \ldots, \lambda_{d} $, where $\lambda_{i}$ are the Lyapunov exponents of the expanding linear endomorphism $g$ and $\lambda^{c}(x,y) = \lim\frac{1}{n}\log ||\frac{\partial F^{n}(x,y)}{\partial y}|| = \lim\frac{1}{n}\sum_{i = 1}^{n-1}\psi(F^{i}(x,y))$ for $\mu-$almost every $(x,y)$ and $\psi : = \log||{\partial F^{n}}{y}||$. The Margulis-Ruelle's inequality can be read as: $$h_{\mu}(F) \leq \sum_{i = 1}^{d}\lambda_{i} + \max\{0 , \lambda^{c}\}.$$

Now for non-ergodic probabilites $\mu$, we define the average central Lyapunov exponent by $\lambda^{c}(\mu) = \int \log ||\frac{\partial F}{\partial y}|| d\mu.$

\

\subsection{Expanding linear endomorphism}

Let $\mathbb{T}^{d}$ be the $d-$torus and let $A$ be a $d\times d$ matrix with integer coefficients and such that all its eigenvalues $a_{1},\ldots, a_{d}$ have modulus strictly greater than 1. Then there exists a unique local diffeomorphism $g_{A}: \mathbb{T}^{d} \rightarrow \mathbb{T}^{d}$ such that $g_{A} \circ \pi = \pi \circ A$, where $\pi : \R^{d} \rightarrow \mathbb{T}^{d}$ is the universal covering. Moreover, $g_{A}$ is an expanding dynamics, its Lyapunov exponents are $\lambda_{i} = \log |a_{i}|$, the Lebesgue measure $Leb$ is $g-$invariant and $h_{top}(g_{A}) = h_{Leb}(g_{A}) =  \sum_{i=1}^{d}\lambda_{i}$. (See e.g. \cite{OV16}).

\
\subsection{Transfer Operator}

In this section, we recall some properties of the transfer operator. For more details on the transfer operator see e.g. \cite{S12} or \cite{PU10}.

In what follows, given $T : E \rightarrow E$ a bounded linear operator, we denote the spectral radius of $T$ by $\rho(T)$.

\begin{definition}
Let $F : M \rightarrow M$ be a local homeomorphism on a compact and connected manifold. Given a complex continuous function $\phi: M \rightarrow \C$ , define the Ruelle-Perron-Frobenius operator or transfer operator $\mathcal{L}_{ \phi}$ acting on functions $g : M \rightarrow \C$ this way:
 $$
 \mathcal{L}_{\phi}(g)(x) := \sum_{f(y) = x}e^{\phi(y)}g(y). 
 $$
\end{definition}

Suposse that $\Lo_{\phi}|_{C^{r}}$ has the spectral gap property. On the one hand,
via Mazur's Separation Theorem, the transfer operator has $\rho(\Lo_{\phi}|_{C^0})$ as an eigenvalue for its dual operator, that is, there exists a probability $\nu_\phi$ with $(\Lo_{\phi}|_{C^{0}})^* \nu_{\phi} = \rho(\Lo_{\phi}|_{C^0}) \nu_\phi$. On the other hand,
$\rho(\Lo_{\phi}|_{C^0}) = \rho(\Lo_{f,\phi}|_{C^{r}})$ and $\Lo_{f,\phi}|_{C^{r}}$ admits an eigenfunction $h_\phi \in C^{r}$ with respect to $\rho(\Lo_{f,\phi}|_C^{r})$ which is the leading eigenvalue. We can assume, up to rescaling, that  $\int h_{\phi}d\nu_{\phi} = 1$.
Then the probability $\mu_\phi=h_\phi \cdot \nu_\phi$ is proved to be $f-$invariant and is a candidate for the equilibrium state.

\begin{remark}
We recall the notion of analyticity for functions on Banach spaces. Let $E_{1} , E_{2}$ be Banach spaces and denote by $\mathcal{L}^{i}_{s}(E_{1} , E_{2})$ the space of symmetric $i$-linear transformations from
$E_{1}^{i}$ to $E_{2}$. For notational simplicity, given $P_{i} \in \mathcal{L}^{i}_{s}(E_{1} , E_{2})$ and $h\in E_{1}$ we set
$P_{i}(h) := P_{i}(h,\ldots,h)$.

We say the function $f : U \subset E_{1} \rightarrow E_{2}$, defined on an open subset,
is \emph{analytic} if for all $x \in U$ there exists $r > 0$ and for each $i\ge 1$ there exists $P_{i} \in \mathcal{L}^{i}_{s}(E_{1} , E_{2})$ (depending on $x$) such that
$$
f(x + h) = f(x) + \sum_{i=1}^{\infty}\frac{P_{i}(h)}{i!}
$$ for all $h \in B(0 , r)$ and the convergence is uniform.

Analytic functions on Banach spaces have completely similar properties to real analytic and complex analytic functions.
For instance, if $f : U \subset E \rightarrow F$ is analytic then $f$ is $C^{\infty}$ and for every $x\in U$ one has
$P_{i}=D^{i}f(x)$. For more details see for example \cite[Chapter~12]{C85}. 
\end{remark}

Define $SG := \{\phi \in C^{r}(M , \C) : \mathcal{L}_{f,\phi}|_{C^{r}} \text{ has the spectral gap property} \}$, it follows from \cite[Corollary 4.12]{BC21}

\begin{proposition}\label{analy}
$SG$ is an open subset, and the following map is analytical:
$$
SG \ni \phi \mapsto \big(\rho(\mathcal{L}_{f,\phi|C^{r}}) , h_{\phi} , \nu_{\phi}).
$$
\end{proposition} 

One weaker spectral property is called \textit{quasi-compactness}:

\begin{definition}
Given $E$ a complex Banach space and $T:E \to E$ a bounded linear operator, we say that $T$ is quasi-compact if there exists $0<\sigma<\rho(T)$ and a decomposition of $E=F\oplus H$ as follows: $F$ and $H$ are closed and $T$-invariant, $\dim F < \infty$, $\rho(T|_F)>\sigma$ and $\rho(T|_H) \leq \rho$.
\end{definition}

A definition equivalent to quasi-compactness can be given via the \textbf{essential spectral radius} $$\rho_{ess}(T):=\inf\{r>0; \;sp(L)\setminus \overline{B(0,r)} \text{ contains only eigenvalues of finite multiplicity}\}.
$$
In fact, quasi-compactness is equivalent to having $\rho_{ess}(T) < \rho(T)$.



\

\section{Phase transitions results}\label{Proofs}



This section is devoted to the proof of the Theorem \ref{mainthA}.

The proof is inspired in \cite{BC21}: 1 - Existence of thermodynamic phase transitions; 2 - Existence of spectral phase transition; 3 - Comprehension of the topological pressure function via spectral phase transition.

For the spectral phase transition: 1 - Absence of spectral after thermodynamic phase transition using Nagaev's method; 2 - Spectral gap property before transition using estimates of essential spectral radius.

\subsubsection{Existence of thermodynamic phase transitions}

 Our first goal is to study the topological pressure function, for $t \in \R$:

$$P(t):=\sup\{h_\mu(f) + t\int \phi^{c}d\mu,\; \mu \text{ is an } F-\text{invariant probability} \} = 
$$
$$
\sup\{h_\mu(f) + - t\lambda^{c}(\mu),\; \mu \text{ is an } F-\text{invariant probability} \}.$$
 We will show that $\R \ni t \mapsto P(t)$ is not analytical.
 
 \ 
 
 Initially, we prove the following lemma:

\begin{lemma}\label{L2}
If $F\in \mathcal{D}^{r}$ then $\lambda^{c}(\mu)\geq 0$, for every $\mu\in \mathcal{M}_1(F)$.
\end{lemma}
\begin{proof}
Suppose $\lambda^{c}(\mu)<0$. By the periodic specification property there exist $0<\lambda<1$ and $(x_{0},y_{0})\in \mathbb{T}^{d}\times\mathbb{S}^{1}$ such that $F^{n}(x_{0},y_{0})=(x_{0},y_{0})$ and $||\frac{\partial F^{kn}(x_{0},y_{0})}{\partial y
}||<\lambda^{k}, \forall k\geq 1$. In particular $f(y):=f_{g^{n-1}(x_{0})}\circ f_{g^{n-2}(x_{0})}\circ ...\circ f_{x_{0}}(y)$ satisfies $f(y_{0})=y_{0}$ and $|Df(y_{0})|\leq \lambda$. By intermediate value theorem there exist $\delta>0$ such that $|f(z_1)-f(z_2)|\leq |z_1-z_2|$, for all $z_1,z_2\in B(y_{0},\delta)$. As $F$ is uniformly continuous, given $\epsilon>0$ we can find a $\tilde{\delta}\leq \delta$ such that
\begin{align*}
|f_{g^{m-1}(x_{0})}\circ f_{g^{m-2}(x_{0})}\circ ...\circ f_{x_{0}}(z_1)-f_{g^{m-1}(x_{0})}\circ f_{g^{m-2}(x_{0})}\circ ...\circ f_{x_{0}}(z_2)|<\epsilon
\end{align*}
for all $m\geq 0$ and $z_1,z_2\in B(y_{0},\tilde{\delta})$. As $\epsilon$ is arbitrary this implies that $F$ is not expansive, which is a contradiction.   
\end{proof}

The following result shows us that for $F\in\mathcal{D}^{r}$ the pressure function has thermodynamic phase transition.

\begin{proposition}\label{L3}
There exist $t_0\in (0,1]$ such the pressure function $\mathbb{R}\ni t\mapsto P(t)$ is not analytic in $t_0$. Moreover, $P(t) = h_{top}(g)$ for all $t \geq t_{0}$
\end{proposition}
\begin{proof}
By the previous lemma, we see that $t\rightarrow P(t)$ is monotone, not increasing. We notice that
\begin{itemize}
\item $P(0)=h_{top}(F)\geq h_{top}(g)>0$
\item If $\eta$ is the measure in the definition of $\mathcal{D}^{r}$ then $h_{Leb\times \eta}(F)-t\lambda^{c}(Leb\times\eta)=h_{Leb\times \eta}(F)$
\end{itemize}
Then we have $P(t)\geq h_{Leb\times\eta}(F)\geq h_{Leb}(g)$ for all $t \in \R$. Recall that by the Margullis-Ruelle inequality (see \cite{Rue78}), we have
\begin{align*}
h_{\mu}(F)-\lambda^{c}(\mu)\leq \lambda_{+}(\mu)  -\lambda^{c}(\mu)= \sum\limits_{i=1}^{d}\lambda_{i}(g,Leb).
\end{align*}
Which means that $P(1)\leq \sum\limits_{i=1}^{d}\lambda(g,Leb)=h_{Leb}(g)=h_{top}(g)\leq P(t)$, and so $P(t)=h_{top}(g), \forall t\geq 1$. As $F$ is differentiable then $h_{top}(F)\geq\log(deg(F))> \log(deg(g)) = h_{top}(g).$ We conclude that $P(t)$ is a convex monotone function and $P(0) > P(t) = h_{top}(g)$ for all $t \geq 1$. Thus, defining $t_0=\inf\{t\in (0,1]; P(t)=h_{top}(g)\}$ we have that the topological pressure function $\R \ni t \mapsto P(t)$ is not analytic in $t_{0}$.
\end{proof}

\

\subsubsection{Absence of  spectral gap after transition}

Our next goal is to show that $\mathcal{L}_{F , t\phi^{c}|C^{r-1}}$ has no spectral gap property for $t \geq t_{0}$, where $t_0$ is the thermodynamical phase transition parameter obtained in the Proposition \ref{L3}.

 First we state some important topological properties for $F \in \mathcal{D}^{r}$. As $F$ is topologically conjugated to a uniformly expanding dynamic and its domain is connected, from this conjugacy, we then find that $F$ is expansive, strongly transitive, and it admits generating partition by domains of injectivity. In particular, we can apply Rokhlin's formula (see \cite{OV16}). As we shall see, such properties imply a good understanding of the transfer operator.
 
The following results have the same proof of the respective results in \cite{BC21}:

\begin{lemma}\label{Lemaxi}
Let $f:M \rightarrow M$ be a $C^{r}-$local diffeomorphism  on a compact and connected manifold $M$, such that $\{f^{-n}(x) : n \geq 0\}$ is dense in $M$ for all $x \in M$, and let $\phi \in C^{r}(M , \R)$ be a real  function. If $\Lo_{f,\phi}\varphi=\lambda\varphi$
with $|\lambda|=\rho(\Lo_{f,\phi}|_{C^{r}})$ and $\varphi \in C^{r}(M , \R)\setminus\{0\}$, then
$\Lo_{f,\phi}|\varphi|=\rho(\Lo_{f,\phi}|_{C^{r}})|\varphi|.$ Furthermore, $\rho(\Lo_{f,\phi}|_{C^0})=\rho(\Lo_{f,\phi}|_{C^{r}})$, $\varphi$ is bounded away from zero and $\dim\ker(\Lo_{f,\phi}|_{C^{r}} - \lambda I)=1$.
\end{lemma}

\begin{corollary}\label{medida}
Let $f:M \rightarrow M$ be a $C^{r}-$local diffeomorphism on a compact and connected manifold $M$, such that $\{f^{-n}(x) : n \geq 0\}$ is dense in $M$ for all $x \in M$, and let $\phi \in C^{r}(M , \R)$ be a continuous real function. If $\mathcal{L}_{f, \phi}|_{C^{r}}$ has the spectral gap property then there exists a unique probability $\nu_{\phi}$ on $M$ such that $ (\mathcal{L}_{f, \phi}|_{C^{r}})^{\ast}\nu_{\phi} = \rho(\mathcal{L}_{f, \phi}|_{C^{r}})\nu_{\phi}$. Moreover, $supp(\nu_{\phi}) = M$.
\end{corollary}

\begin{remark}
Given $\Lo_{f,\phi}|_{C^{r}}$ with the spectral gap property, by Lemma \ref{Lemaxi} there exists a unique $h_{\phi} \in C^{r}(M , \C)$ such that $\Lo_{f,\phi} h_{\phi}=\rho(\Lo_{f,\phi}) h_{\phi}$ and $\int h_{\phi} d\nu_{\phi} = 1$. Moreover, denote the $f-$invariant probability $h_{\phi}\nu_{\phi}$ by $\mu_{\phi}$. Note that $supp(\nu_{\phi}) = M$ and $h_{\phi} > 0$, thus, by this lemma, $supp(\mu_{\phi}) = M.$
\end{remark}

\begin{lemma}\label{Lemapress}\label{GapAnalt}
Let $f:M \rightarrow M$ be a $C^{r}-$local diffeomorphism on a compact and connected manifold $M$, such that $\{f^{-n}(x) : n \geq 0\}$ is dense in $M$ for all $x \in M$ and $f$ admits generating partition by domains of injectivity, and let $\phi \in C^{r}(M , \R)$ be a real function. If $\Lo_{f,s\phi}|_{C^{r}}$ has spectral gap for a given $s \in \R$, then:

(i) $P_{top}(f,s\phi)=\log \rho(\Lo_{f,s\phi}|_{C^{r}})$ and $\mu_{s\phi}$ is an equilibrium state of $f$ with respect to $s\phi$;

(ii) $\R \ni t\mapsto P(f,t\phi)$ is analytic on $s$.
\end{lemma}

Let's prove it now the absence of the spectral gap property after the transition; in fact, we will establish a relation between the spectral gap property of the transfer operator and the strict convexity of the topological pressure function. The proof uses the ideas of the respective result in \cite{BC21}.

\begin{proposition}\label{propat}
Let $F \in \mathcal{D}^{r}$ be for $r \geq 2$. If $\Lo_{F,s\phi^{c}}$ has the spectral gap property on $C^{r-1}$, for some $s \in \R$, then the topological pressure function $t \mapsto P(t)$ is strictly convex in a neighbourhood of $s$. In particular, $\Lo_{F , t\phi^{c}}|_{C^{r-1}}$ does not have spectral gap property for all $t \geq  t_0$.
\end{proposition}
\begin{proof}

By Proposition \ref{analy}, spectral gap holds for any small perturbation of $\Lo_{F , s\phi^{c}}|_{C^{r-1}}$ and such operator varies analytically with the potential. Let $\lambda_{\phi}$ be denoting $\rho(\Lo_{f,\phi}|_{C^{r-1}})$ and fix $\psi \in C^{r-1}$. Then the function $\Upsilon(t):=\dfrac{\lambda_{s\phi^{c} + t \psi}}{\lambda_{s\phi^{c}}}$ is analytic for $t \in \R$ in a small neighborhood of zero.

By Nagaev's method we have $\sigma^2=\Upsilon''(0)$,
where $\sigma^2:=\sigma^2_{f,s\phi^{c}}(\psi)$ is the variance of the \textit{Central Limit Theorem} with respect to the map $F$, probability $\mu_{s\phi^{c}}$ and observable $\psi$. Moreover $\sigma = 0$ if, only if,  there exists $a \in \R$ and $u \in C^{r-1}$ such that 
$$\psi(x)= a + u\circ f(x) - u(x),$$
 for $\mu_{s\phi^{c}}-\text{a.e. } x$ (see \cite[Lecture 4]{S12} for more details). 

  In particular, we take $\psi = \phi^{c} - \int \phi^{c} d\mu_{s\phi^{c}}$. Now define the following pressure function $G(t) := P_{top}( f , s\phi^{c} + t\psi)$, then we have
 $$G(t) = P(t + s) - t\int \phi^{c} d\mu_{s\phi^{c}},$$
  for all $t \in \R$, thus $P''(s)=G''(0)$. Furthermore, for $t \in \R$ close enough to zero we have that $\mathcal{L}_{f , s\phi^{c} + t\psi}$ has the spectral gap property and thus $G(t) = \log \lambda_{s\phi^{c} + t\psi}$, by Lemma \ref{Lemapress}. Hence:
$$
\sigma^{2} = \sigma^2_{f,s\phi^{c}}(\psi)= \Upsilon''(0) = \dfrac{ \dfrac{d^{2}\,\lambda_{s\phi^{c}  + t \psi}}{d^{2} t}|_{t = 0} }{\lambda_{s\phi^{c}} } = \big(G'(0)\big)^{2} + G''(0).
$$

\;

\noindent As $F$ is expansive, it follows from \cite{W92} that $G'(0) = \int \psi \,d\mu_{s\phi^{c}} = 0$. Then we conclude that $\sigma^{2} = G''(0)= P''(s)$.

Suppose by contradiction that $P''(s) = 0$. Hence, by Nagaev's method, there exists $a \in \R$ and $u \in C^{r-1}
$ such that 

$$-\phi^{c}(x , y) = \log |Df_{x}(y)|= a + u\circ F(x,y) - u(x,y) , \text{ for } \mu_{s\phi^{c}}-\text{a.e. } (x,y),$$
where $\mu_{s\phi^{c}}$ is the equilibrium state obtained via the spectral gap property, which we know has full support. Therefore the Birkhoff time average  $\frac{1}{n} S_n \phi^{c}$ converges uniformly to $-a$. Consequently, $a$ is the unique central Lyapunov exponent of $F$. Observe that $a= 0$ because, by definition of $\mathcal{D}^{r}$, $F$ has some central Lyapunov exponent zero. Hence $\log |Df_{x}(y)|= u\circ F(x,y) - u(x,y)$. In particular the unique Lyapunov exponent of $f_{0} : \Sc^{1} \rightarrow \Sc^{1}$ is zero, which implies  that $f_{0}$ is invertible. This is absurd because $deg(F) > deg(g)$. 

We conclude that $P''(s) = G''(0) > 0$. In fact, since $P$ is analytic then $P''(t)>0$ for $t$ close to $s$. This implies that $P$ is strictly convex in a neighbourhood of $s$.

Lastly, by Proposition \ref{L3} the topological pressure function $t \to P(t)$ is constant in $[t_{0} , +\infty)$. Therefore $\Lo_{F , t\phi^{c}}|_{C^{r-1}}$ does not have spectral gap property for all $t \geq  t_0$.
\end{proof}

\

\subsubsection{Spectral gap before of transition}

Our next goal is show that $\mathcal{L}_{F , t\phi^{c}|C^{r-1}}$ has spectral gap property for $t < t_{0}$.

First, we observe that for local diffeomorphisms with dense pre-images, if the transfer operator is quasi-compact then it has the spectral gap property. The proof is analogous to the respective result in \cite{BC21}, 

\begin{proposition}\label{LemaEss}
Let $f:M \rightarrow M$ be a $C^{r}-$local diffeomorphism  on a compact and connected manifold $M$, such that $\{f^{-n}(x) : n \geq 0\}$ is dense in $M$ for all $x \in M$, and let $\phi \in C^{r}(M , \R)$ be a continuous real function. If $\Lo_{f,\phi}|_E$ is quasi-compact, then it  has the spectral gap property.
\end{proposition}

The following results tell us that before the phase transition $t_0$, the topological pressure function is strictly decreasing.

\begin{lemma}\label{L7}
 If $F \in \mathcal{D}^{r}$, with $r \geq 2$, then the function $\R \ni t\mapsto P(t)$ is strictly decreasing in $(-\infty,t_0)$.
\end{lemma}
\begin{proof}
Suppose by contradiction that there is an interval $(a,b)\in (-\infty,t_0)$ such that $P(t)=C$, for all $t\in (a,b)$. Fix $t \in (a , b)$ and take $w<0$ small such that $t+w\in (a,b)$. As $F$ is expansive there exists $\mu_{t\phi^{c}}$ equilibrium state of $F$ with respect to $t\phi^{c}$ and we have
\begin{align*}
h_{\mu_{t\phi^{c}}}(F)-(t+w)\lambda^{c}(\mu_{t\phi^{c}})\leq P(t+w)=P(t)=h_{\mu_{t\phi^{c}}}(F)-t\lambda^{c}(\mu_{t\phi^{c}})
\end{align*}
which means that $w\lambda^{c}(\mu_{t\phi^{c}})\geq 0$, and $\lambda^{c}(\mu_{t\phi^{c}})\leq 0$. However $\lambda^{c}(\mu_{t\phi^{c}})\geq 0$ by Lemma \ref{L2},  thus  $\lambda^{c}(\mu_{t\phi^{c}})=0$, and so $P(t)=h_{\nu}(F)$, for all $t\in(a,b)$ and $h_{\nu}(F)=\max\{h_{\mu}(F); \lambda^{c}(\mu)=0\}$. In conclusion, using Margulis-Ruelle inequality, we have:

\begin{align*}
h_{top}(g)=\sum\limits_{i=1}^{d} \lambda_{i}(g, Leb)\geq h_{\nu}(F)\geq h_{Leb\times\eta}(F)\geq h_{top}(g).
\end{align*}
In that way, we have $a\geq t_0$, by definition of $t_0$, which is a contradiction.
\end{proof}

\

For the following lemmas, we need to estimate the essential spectral radius. The results \cite[Theorems 1 and 2]{CL97}, in our context, summarise to:

\begin{theorem}\label{Lat}
Assume that $f:M\rightarrow M$ is any smooth covering and $\phi \in C^{r}(M, \R)$ then
$$ \rho_{ess}(\Lo_{f,\phi}|_{C^k}) \leq \exp\big[\sup_{\mu \in \mathcal{M}_{e}(f)}\{h_\mu(f)+\int\phi \text{d}\mu-k\lambda_{\min}(f , \mu)\}\big] \text{ and }$$
$$ \rho(\Lo_{f,\phi}|_{C^k}) \leq \exp\big[\sup_{\mu \in \mathcal{M}_{1}(f)}\{h_\mu(f)+\int\phi \text{d}\mu\}\big],$$
for $k=0,1,\dots, r.$ Where $\lambda_{\min}(f , \mu)$ denotes the smallest Lyapunov-Oseledec exponent of $\mu$.
\end{theorem}


In particular:

\begin{lemma}\label{essest}
 If $F \in \mathcal{D}^{r}$  and $k=0,1,\dots, r$  then
 $ \rho_{ess}(\Lo_{F,t\phi^{c}}|_{C^k}) < e^{P(t)},$ for all $ t < t_{0}$. 
\end{lemma}
\begin{proof}
Let $\mu$ be a $F-$invariant and ergodic probability. Suppose that $\lambda_{\min}(\mu)=\lambda^{c}(\mu)$. Then we have
\begin{align*}
h_{\mu}(F)+ t\int\phi^{c} d\mu-k\lambda_{\min}(\mu)=h_{\mu}(F)-(t+k)\lambda^{c}(\mu)\leq P(t+k)<P(t)
\end{align*}
for $t<t_0$, by the previous lemma. Suppose now that $\lambda_{\min}(\mu) \neq \lambda^{c}(\mu)$, in this case $\lambda_{\min}(\mu) = \lambda_{1}(g , Leb)> 0$. Then:
$$
h_{\mu}(F)+ t\int\phi^{c} d\mu-k\lambda_{\min}(\mu)=h_{\mu}(F)-t\lambda^{c}(\mu)-k\lambda_{\min}(\mu)\leq
$$
$$ P(t)-k\lambda_{1}(g , Leb)<P(t).
$$
The result now follows from the previous Theorem.
\end{proof}

We start proving spectral gap for non-positive $t$.
\begin{lemma}
If $F \in \mathcal{D}^{r},$ with $r \geq2$, then $\Lo_{F,t\phi^{c}}|_{C^{r-1}}$ has spectral gap property for all $t\leq 0.$
\end{lemma}
\begin{proof}

\textit{Claim: $\Lo_{F,0}|_{C^{r-1}}$ has spectral gap.}

\

\noindent In fact, 
by Lemma \ref{essest}:
$$\rho_{ess}(\Lo_{F,0}|_{C^{r-1}})<  e^{P(0)}=e^{h_{top}(F)}=\deg(F)=\rho(\Lo_{F,0}|_{C^{r-1}}).$$
Hence $\Lo_{f,0}|_{C^{r-1}}$ is quasi-compact and by Lemma \ref{LemaEss} it has spectral gap property.

\

Let $\tau_1:=\inf\{t<0; \Lo_{F,t\phi^{c}}|_{C^{r-1}} \text{ has the spectral gap property}\}$ be. Suppose by contradiction that $\tau_1 > -\infty.$ Since the spectral gap property is open (see Corollary \ref{analy}) then $\Lo_{F,\tau_{1}\phi^{c}}|_{C^{r-1}}$ has no spectral gap property.

\

\textit{Claim: $\rho(\Lo_{F,\tau_1 \phi^{c}}|_{C^{r-1}})=e^{P(\tau_1)}$.}

\

\noindent In fact, for all $t \in (\tau_1,0]$ the transfer operator $\Lo_{t\phi^{c}}|_{C^{r-1}}$ has the spectral gap property, then $\rho(\Lo_{t\phi^{c}}|_{C^{r-1}})=e^{P(t)}$ and $(\tau_1,0]\ni t \mapsto \rho(\Lo_{-t\log|Df|}|_{C^{-1}})$ is decreasing, by Lemma \ref{Lemapress} and Lemma \ref{L7}. 
Now suppose by contradiction there is a $t_n \searrow \tau_1$ such that $\rho(\Lo_{\tau_1\phi^{c}}|_{C^{r-1}}) < \rho(\Lo_{t_n\phi^{c}}|_{C^{r-1}})-\varepsilon$. By semi-continuity of spectral components (see e.g. \cite{K95}) we have  $$\rho(\Lo_{\tau_1\phi^{c}}|_{C^{r-1}}) \geq \sup_{t\in(\tau_1,0]}\rho(\Lo_{t\phi^{c}}|_{C^{r-1}})=\sup_{t\in(\tau_1,0]}e^{P(t)}=e^{P(\tau_1)}.$$

\noindent On the other hand, by Theorem \ref{Lat} $\rho(\Lo_{\tau_1\phi^{c}}|_{C^{r-1}})\leq e^{P(\tau_1)}$. Thus we conclude $\rho(\Lo_{\tau_1 \phi^{c}|}|_{C^{r-1}})=e^{P(\tau_1)}$.

\

Therefore, 
$$\rho(\Lo_{\tau_1\phi^{c}}|_{C^{r-1}})=e^{P(\tau_1)}> \rho_{ess}(\Lo_{\tau_1\phi^{c}}|_{C^{r-1}}),$$
and by Lemma \ref{LemaEss} $\Lo_{\tau_1\phi^{c}}|_{C^{r-1}}$ has the spectral gap property, contradicting the minimality of $\tau_1.$
\end{proof}

Given $F \in \mathcal{D}^{r},$ with $r \geq2$, define analogously
$\tau_2:=\sup\{t>0; \Lo_{F,t\phi^{c}}|_{C^{r-1}}$ has the spectral gap property $\}$, note that $\tau_2$ exists and is at most $t_0$ by Proposition \ref{propat}.
Analogously, it holds $\rho(\Lo_{F,\tau_2\phi^{c}}|_{C^{r-1}})=e^{P(\tau_2)}$. Since the spectral gap property is open, $\Lo_{F,\tau_2\phi^{c}}|_{C^{r-1}}$ has no spectral gap property and thus $\rho(\Lo_{F,\tau_2\phi^{c}}|_{C^{r-1}})=\rho_{ess}(\Lo_{F,\tau_2\phi^{c}}|_{C^{r-1}})$ by Lemma \ref{LemaEss}.
\begin{lemma}
$\tau_2=t_0$
\end{lemma}
\begin{proof}
Suppose by absurd that $\tau_{2} < t_{0}$. By Lemma \ref{essest}, we get:
$$e^{P(\tau_2)}=\rho(\Lo_{\tau_2\phi^{c}}|_{C^{r-1}})=\rho_{ess}(\Lo_{\tau_2\phi^{c}}|_{C^{r-1}})< e^{P(\tau_2)}.$$
Which is absurd.
\end{proof}

From definition of $\tau_2$ we get:

\begin{corollary}
If $F \in \mathcal{D}^{r},$ with $r \geq2$, then $\Lo_{F,t\phi^{c}}|_{C^{r-1}}$ has spectral gap for all $t<t_0$.
\end{corollary}

The previous Corollary, together with the Proposition \ref{L3}, Lemma \ref{Lemapress} and Proposition \ref{propat}, completes the proof of the Theorem \ref{mainthA}.

\

\section{Applications}\label{ProofsB}

This section is devoted to the proof of the results of multifractal analysis ( Theorem \ref{mainthA}, and Corollary \ref{mainthD}). However, we will need results like the large deviations principle, that have their importance independently.






Throughout this section we will suppose $F \in \mathcal{D}^{r}$.  Our results will  also be valid for $f : \Sc^{1} \rightarrow \Sc^{1}$ a $C^{1}-$local diffeomorphism with $Df$ Holder continuous, non-invertible nor expanding, in view that $f$ is conjugated to an expanding dynamics (\cite{CM86}) and it holds an analogous of Theorem \ref{mainthA} (\cite{BC21}).

\subsection{Large deviations principle}\label{secLDP}


In our context $\lim \frac{1}{n}\log |\frac{\partial F^{n}}{\partial y}(w)| = \lambda^{c}(w)$ for  $\mu_{0}-$almost every $w$, where $\mu_{0}$ is the measure of maximum entropy of $F$. Thus, we are interested in understanding asymptotically 
$$\frac{1}{n}\log \mu_{0}\Big(\{w \in \mathbb{T}^{d} \times \Sc^{1}: \frac{1}{n}\log \Big|\frac{\partial F^{n}}{\partial y}(w)\Big| \in [a , b]\} \Big),$$
with $[a , b]$ intersecting the central Lyapunov spectra 
$$L^{c}(F) := \{\alpha \in \R :  \exists w \in \Sc^{1} \times T^{d} \text{ with } \lim\frac{1}{n}\log \Big|\frac{\partial F^{n}}{\partial y}(w)\Big| = \alpha\}.$$
Note that since  $F$ has specification property and $$\frac{1}{n}\log \Big|\frac{\partial F^{n}}{\partial y}(w)\Big| = \frac{1}{n}\sum_{j = 1}^{n-1}\log\Big|\frac{\partial F}{\partial y}(F^{j}(w))\Big|$$ then the central Lyapunov spectra satisfies
$$L^{c}(F) = \Big\{\int \log\Big|\frac{\partial F}{\partial y}\Big| d\nu : \nu \in \mathcal{M}_{1}(F)\Big\} = \overline{\Big\{\int \log\Big|\frac{\partial F}{\partial y}\Big| d\nu : \nu \in \mathcal{M}_{e}(F)\Big\}}$$ and is a compact interval (see \cite{Daniel}). In fact, $L^{c}(F)$ will have non-empty interior and by Lemma \ref{L2} we have $\inf L^{c}(F) = 0$. 

Define $\lambda^{c}_{\min} := \inf_{t < t_{0}}\lambda^{c}(\mu_{t\phi^{c}})$ and $\lambda^{c}_{\max} := \sup_{t < t_{0}}\lambda^{c}(\mu_{t\phi^{c}})$, where $\mu_{t\phi^{c}}$ is the equilibrium state of $F$ with respect to $t\phi^{c}$ obtained in Lemma \ref{GapAnalt}. We will improve the description of $L^{c}(F)$ using the equilibrium states $\mu_{t\phi^{c}}$:

\begin{lemma}
$L^{c}(F) = [0 , \lambda^{c}_{\max}] $ and $\lambda^{c}_{\min} \in L^{c}(F)$.
\end{lemma}
\begin{proof}
Since  $F$ has the specification property, applying \cite{Daniel}, we have that $L^{c}(F) = \{\int \log |\frac{\partial F}{\partial y}| d\nu : \nu \in \mathcal{M}_{1}(F)\}$. In particular, $\lambda^{c}(\mu_{t\phi^{c}}) \in L^{c}(F)$ for all $t < t_{0}$ and $\lambda^{c}_{\max} , \lambda^{c}_{\min} \in L^{c}(F).$
We will show that $\sup L^{c}(F) = \lambda^{c}_{\max}$. Take a probability $\nu \in \mathcal{M}_{1}(F)$ and $t < 0$. Then:
$$
h_{\nu}(F) -  t\int \log \Big|\frac{\partial F}{\partial y}\Big| d\nu \leq h_{\mu_{t\phi^{c}}} - t\int  \log \Big|\frac{\partial F}{\partial y}\Big| d\mu_{t\phi^{c}} \Rightarrow \frac{h_{\nu}(F)}{-t} + \int \log \Big|\frac{\partial F}{\partial y}\Big| d\nu \leq 
$$
$$
\frac{h_{\mu_{t\phi^{c}}}(F)}{-t} + \int \log \Big|\frac{\partial F}{\partial y}\Big| d\mu_{t\phi^{c}} \leq \frac{\log deg(F)}{-t} + \int \log \Big|\frac{\partial F}{\partial y}\Big| d\mu_{t\phi^{c}} ,
$$
taking $t \mapsto -\infty$ we have that $\int \log |\frac{\partial F}{\partial y}| d\nu \leq \sup_{t < t_{0}} \int \log |\frac{\partial F}{\partial y}| d\mu_{t\phi^{c}}$. Thus, $\sup L^{c}(F) = \lambda^{c}_{\max}$ and conclude that $L^{c}(F) = [0 , \lambda^{c}_{\max}] $.
\end{proof}





Define the free energy  $$\mathcal{E}(t):= \limsup_{n \mapsto +\infty} \frac{1}{n}\log \int e^{-tS_{n}\phi^{c} }d\mu_{0},$$ where $t \in \R$ and $S_{n}\psi := \sum_{j=1}^{n-1}\psi \circ F^{j}$ is the usual Birkhoff sum.  In our setting we will prove that the limit above does exist for $t  > - t_{0}$, where $t_{0}$ is the transition parameter, and $\mathcal{E}$ has good properties.

\begin{lemma}\label{prop:free.energy}
For all $t > -t_{0}$ the following limit exists
$$
\cE(t)
	=\lim_{n\to\infty} \frac1n \log \int e^{-t S_n \phi^{c}} \; d\mu_{0}
	= P(-t) -\log deg(F)
$$
Moreover, 
$\cE : (-t_{0} , +\infty ) \rightarrow \R$ is real analytic and strictly convex. \end{lemma}
\begin{proof}
The prove that guarantees $$\cE(t)
	=\lim_{n\to\infty} \frac1n \log \int e^{-t S_n \phi^{c}} \; d\mu_{0} = \log\rho(\mathcal{L}_{F,-t\phi^{c}}) - \log\rho(\mathcal{L}_{F , 0})$$ is analogous to \cite[Proposition  5.2]{BCV16}, using that $\mathcal{L}_{F,t\phi^{c}}$ has spectral gap property for $t < t_{0}$ (see Theorem \ref{mainthA}). On the other hand, $\log\rho(\mathcal{L}_{F,t\phi^{c}}) = P(t)$ because $\mathcal{L}_{F,t\phi^{c}}$ has spectral gap property and we can apply Lemma \ref{GapAnalt}.
	Finally, since $P(t)$ is analytic and strictly convex in $(-\infty , t_{0})$ by Theorem \ref{mainthA}, we have finished the proof of the lemma.
 \end{proof}

Since the function
$\cE : (-t_{0} , +\infty ) \rightarrow \R$ is strictly convex it is well defined the ``local"
Legendre transform $I$ given by
\begin{equation*}
I(s)
	:= \sup_{-t_{0} < t < +\infty} \; \big\{ st-\cE(t) \big\}.
\end{equation*}

It is well known that the strict convexity and differentiability of the free energy $\mathcal{E}$ imply that the domain of $I$ is $\{\mathcal{E}'(t) : t >  - st_{0}\}$, and $I$ is strictly convex and non-negative. In fact, it is not hard to check the variational property
$
I (\cE'(t) )
	= t \cE'(t) - \cE(t).
$
Moreover, $I(s)=0$ if and only if $s=\lambda^{c}(\mu_{0})$. Thus
$I$ will be a strictly convex and real analytic function.

\begin{remark}
Given $s < t_{0}$, by Theorem \ref{mainthA} and openness of the spectral gap property, $\mathcal{L}_{F, s\phi^{c} + \psi}$ has the spectral gap property on $C^{r-1}$ for $\psi \in C^{r-1}$ and $||\psi||_{r-1} < \epsilon$. Applying Lemma \ref{GapAnalt} we have that $B(0 , \epsilon) \ni \psi \mapsto P_{top}(F , s\phi^{c} +\psi) $ is analytic.  By \cite{W92}, we have $P'(t) = \int \phi^{c} d\mu_{t\phi^{c}}$.
\end{remark}

It follow from the previous remark that the domain of $I$ is equal to $\{-\int\phi^{c}d\mu_{-t\phi^{c}} : t > t_{0}\} = (\lambda^{c}_{\min} , \lambda^{c}_{\max})$.

We collect all of these statements in the following:

\begin{lemma}\label{lemmaI}
\begin{itemize}
\item[i)] The domain if $I$ is $(\lambda^{c}_{\min} , \lambda^{c}_{\max})$;
\item[ii)] $I$ is a positive, strictly convex function and $I(s)= 0$ if and only $s=\lambda^{c}(\mu_{0})$;
\item[iii)] $I$ is a real analytic function.
\end{itemize}
\end{lemma}

 The following figure describes the expected shape of the Legendre transform $I$:
 
 \;

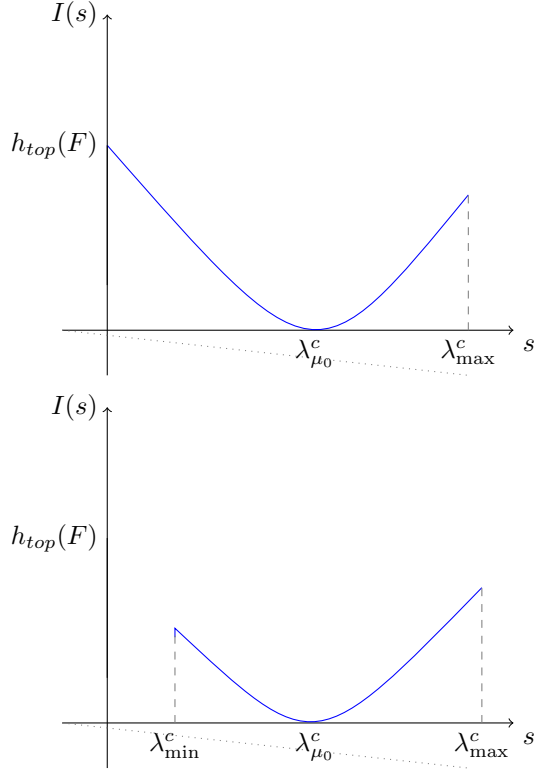
\begin{figure}[h]
 \centering
\begin{tikzpicture}[scale=0.6,domain=0:3]
\draw[->] (-1,0) -- (9,0) node[below right] {$s$};
\draw[->] (0,-1) -- (0,7) node[left] {$I(s)$};

 \draw[dotted,gray] (-1,0) -- (8,-1);
  \draw[blue] (0,4.1) .. controls (4.6,-1.15) .. (8,3);
\draw (8,0) node [below]{$\lambda_{\max}^c$};
\draw (0,4.1) node[left] {$h_{top}(F)$} -- (0,1);
\draw (4.6,0) node[below] {$\lambda_{\mu_0}^c$};
\draw[dashed,color=gray] (8,0) -- (8,3);
\end{tikzpicture} %
    \qquad
\begin{tikzpicture}[scale=0.6,domain=0:3]
\draw[->] (-1,0) -- (9,0) node[below right] {$s$};
\draw[->] (0,-1) -- (0,7) node[left] {$I(s)$};

 \draw[dotted,gray] (-1,0) -- (8,-1);
  \draw[blue] (1.5,1.9) -- (1.5,2.1).. controls (4.6,-.8) .. (8.3,3);
\draw (1.5,0) node[ below]{$\lambda_{\min}^{c}$};
\draw (8.3,0) node [below]{$\lambda_{\max}^c$};
\draw (0,4.1) node[left] {$h_{top}(F)$} -- (0,1);
\draw (4.6,0) node[below] {$\lambda_{\mu_0}^c$};
\draw[dashed,color=gray] (1.5,0) -- (1.5,2.1);
\draw[dashed,color=gray] (8.3,0) -- (8.3,3);
\end{tikzpicture} %
    \caption{ { The Legendre transform $I$}}%
    \label{hattau}%
\end{figure}

The following results are obtained from the Gartner-Ellis theorem (see e.g.~\cite{DZ98,RY08}) as a consequence
of the differentiability of the free energy function.

\begin{proposition}\label{thm:LDP.Ellis}
Given any interval $[a,b]\subset (\lambda^{c}_{\min} , \lambda^{c}_{\max})$ it holds that
$$
\lim_{n\to\infty} \frac1n \log \mu_{0}
	\left(w\in \cS^{1} \times T^{d} : \frac{1}{n}\log |\frac{\partial F^{n}}{\partial y}(w)| \in [a,b] \right)
	=-\inf_{s\in[a,b]} I(s).
$$
\end{proposition}

Since $F$ is expansive and has the specification property, applying \cite{Bo74} we have that $\mu_{0}$ is a Gibbs probability with respect to dynamics $F$ and null potential. Thus, applying \cite{Y90} the measure of maximum entropy $\mu_{0}$ satisfies a large deviations principle for all continuous observables, with rate function obtained through thermodynamic quantities. In particular it holds:

\begin{proposition}\label{LDP:thermodynamical}
Given any interval $[a,b]\subset L^{c}(F)$ it holds that
$$
\limsup_{n\to\infty} \frac1n \log \mu_{0}
	\left(w\in \cS^{1} \times T^{d} : \frac{1}{n}\log |\frac{\partial F^{n}}{\partial y}(w)| \in [a,b] \right)
	\le 
	$$
	$$
	- h_{top}(F) + \sup\{h_{\nu}(F) :  \lambda^{c}(\nu) \in [a,b]\}
$$
and
$$
\liminf_{n\to\infty} \frac1n \log \mu_{f,\phi}
	\left(w\in \cS^{1} \times T^{d} : \frac{1}{n}\log |\frac{\partial F^{n}}{\partial y}(w)| \in (a,b) \right)
	\ge
	$$
	$$
	- h_{top}(F) + \sup\{h_{\nu}(F) :  \lambda^{c}(\nu) \in [a,b]\}
$$
\end{proposition}

\begin{remark}\label{rem1}
Since  $F$ is expansive, then $\nu \mapsto h_{\nu}(F)$ is upper semicontinuous (see e.g. \cite{OV16}). By contraction principle, we know that there is the uniqueness of the large deviations rate function (see e.g. \cite{DZ98}). Thus, it follows from the two previous propositions that $I(s) = h_{top}(F) - \sup\{h_{\nu}(F) : \lambda^{c}(\nu) = s\}$, for all $s \in (\lambda^{c}_{\min} , \lambda^{c}_{\max}).$
\end{remark}

\subsection{Multifractal Analysis}\label{chapAM}

In this section will prove Theorem \ref{mainthC} and Corollary \ref{mainthC}.

\subsubsection{Entropy spectrum of central Lyapunov exponents }

In our context, we are interested in studying the following fractal sets created from central Lyapunov exponents:
$$
L^{c}_{a , b} =\{w\in T^{d} \times \Sc^{1}; \lim\frac{1}{n}\log |\frac{\partial F^{n}(w)}{\partial y}| \in [a , b]\} \text{ and }$$
$$
E_{a,b} := \Big\{w\in T^{d} \times \Sc^{1} :  \nexists \lim\frac{1}{n}\log |\frac{\partial F^{n}(w)}{\partial y}| \text{ and } 
$$
$$\limsup_{n\to\infty} \frac{1}{n}\log |\frac{\partial F^{n}(w)}{\partial y}| \in  [a,b] \text{ or } \liminf_{n\to\infty} \frac{1}{n}\log |\frac{\partial F^{n}(w)}{\partial y}| \in  [a,b]\Big \}
$$

\

 
 \begin{lemma}\label{lemX}
 Given $[a , b] \subset L^{c}(F)$ we have that $h_{L^{c}_{a , b}}(F) 
 = \sup\{h_{\nu}(F) :  \lambda^{c}(\nu) \in [a,b ]\}.$ Moreover, if $a < b$ then $h_{L^{c}_{a , b}}(F) 
 = h_{E_{a,b}}(F)$.
\end{lemma}
 \begin{proof}
 Define $LD_{c,d} := \limsup_{n\to\infty} \frac1n \log \mu_{0}
	\left(w\in \cS^{1} \times T^{d} : \frac{1}{n}\log |\frac{\partial F^{n}}{\partial y}(w)| \in [c,d] \right)$. As previously mentioned, $\mu_{0}$ is a Gibbs probability with respect to dynamics $F$ and null potential. Applying \cite[Theorem A]{BV17} we have that $h_{L^{c}_{a , b}}(F) \leq  h_{top}(F) + LD_{a-\delta,b+\delta}$, for all $\delta >0$. By Proposition \ref{LDP:thermodynamical} 
 $$LD_{a-\delta,b+\delta} \leq - h_{top}(F) + \sup\{h_{\nu}(F) :  \lambda^{c}(\nu) \in [a-\delta,b +\delta]\}.$$
 Since  $\nu \mapsto h_{\nu}(F)$ is upper semicontinuous, thuas making $\delta$ go to zero, $h_{L^{c}_{a , b}}(F) \leq \sup\{h_{\nu}(F) :  \lambda^{c}(\nu) \in [a,b ]\} $.
 Similarly, if $a < b$ then
 $h_{E_{a,b}}(F) \leq \sup\{h_{\nu}(F) :  \lambda^{c}(\nu) \in [a,b ]\}$.

 On the other hand, by \cite{TV03} we know that $h_{L^{c}_{d , d}}(F)  = \sup\{h_{\nu}(F) :  \lambda^{c}(\nu) = d\}.$ The upper semicontinuity of $\nu \mapsto h_{\nu}(F)$ also guarantees that there exists $d \in [a , b]$ with  $\sup\{h_{\nu}(F) :  \lambda^{c}(\nu) \in [a , b]\} = \sup\{h_{\nu}(F) :  \lambda^{c}(\nu) = d\}.$ Hence:
 $$
 \sup\{h_{\nu}(F) :  \lambda^{c}(\nu) \in [a , b]\} = \sup\{h_{\nu}(F) :  \lambda^{c}(\nu) = d\} =$$
 $$ h_{L^{c}_{d , d}}(F) \leq h_{L^{c}_{a , b}}(F)
 \leq \sup\{h_{\nu}(F) :  \lambda^{c}(\nu) \in [a , b]\}.
 $$
 We conclude that $h_{L^{c}_{a , b}}(F) = \sup\{h_{\nu}(F) :  \lambda^{c}(\nu) \in [a,b ]\}$. 
 
 Finally, suppose that $a < b$. Let $\tilde{\nu}_{1} \in \mathcal{M}_{1}(F)$ be  such that $h_{\tilde{\nu}_{1}}(F) = \sup\{h_{\nu}(F) :  \lambda^{c}(\nu) \in [a,b ]\}$ and $\lambda^{c}(\tilde{\nu}_{1}) \in [a , b]$. Fix $\mu \in \mathcal{M}_{1}(F)$ such that $\lambda^{c}(\mu) \in (a , b)$ and $\lambda^{c}(\mu) \neq \lambda^{c}(\tilde{\nu}_{1})$. Defining $\mu_{t} := (1-t)\tilde{\nu}_{1} + t\mu$ and taking $t$ small enough, we can find $\tilde{\nu}_{2} \in \mathcal{M}_{1}(F)$ such that $|h_{\tilde{\nu}_{1}}(F) - h_{\tilde{\nu}_{2}}(F)| < \gamma$, $\lambda^{c}(\tilde{\nu}_{2}) \in [a,b]$ and $\lambda^{c}(\tilde{\nu}_{2}) \neq \lambda^{c}(\tilde{\nu}_{1})$.
 
  Since  $F$ has the specification property it is well-known that for any 
$F$-invariant probability measure $\mu$ there exists a sequence of $F$-invariant ergodic probability
measures $\mu_n$ so that $\mu_n\to\mu$ in the weak$^*$ topology and $h_{\mu_n}(f) \to h_\mu(f)$
as $n\to\infty$ (c.f.  \cite[Theorem B]{EKW}). 
 
 Thereby we can find $F-$invariant and ergodic probabilities $\nu_{1}$ and $\nu_{2}$ such that:
 \begin{itemize}
     \item[i)] $\lambda^{c}(\nu_{i}) \in [a , b]$, for $i = 1, 2$;
     \item[ii)]   $\lambda^{c}(\nu_{1}) \neq \lambda^{c}(\nu_{2})$;
     \item[iii)] $|h_{\tilde{\nu}_{1}}(F) - \sup\{h_{\nu}(F) :  \lambda^{c}(\nu) \in [a,b ]\}| < 2\gamma$.
 \end{itemize}
 
 Now the proof follows the same lines as the proof of Theorem~2.6 in \cite{Daniel}. Define $\psi : =\log |\frac{\partial F^{n}}{\partial y}| .$
Consider a strictly decreasing sequence $(\delta_k)_{k\ge 1}$ of positive numbers converging
to zero, a strictly increasing sequence of positive integers $(\ell_k)_{k\ge 1}$, so that the sets 
$$
Y_{2k+i}=\Big\{ w \in T^{d}\times \Sc^{1} \colon | \frac1n S_n\psi(w)  -  \int \psi \, d\nu_{i}| < \delta_k 
	\text{ for every } n\ge \ell_k \Big\}
$$
satisfy $\nu_i(Y_{2k+i}) > 1-\gamma$ for every $k$ ($i=1,2$).
Consider the fractal set $\tilde{F}$ given \emph{ipsis literis} by the construction of Subsection~3.1 in \cite{Daniel} with $\nu_i$ replacing $\mu_i$,  $\sup\{h_{\nu}(F) :  \lambda^{c}(\nu) \in [a,b ]\}$ replacing $C$.
From the construction (c.f. Lemma~3.8 in \cite{Daniel}) there is a sequence $(t_k)_{k\ge 1}$ so that
$$
\lim_{k\to\infty} | \frac{1}{t_{2k+i}}  S_{t_{2k+i}} \psi (w) - \int \psi \, d\nu_i | = 0
	\quad \text{for every $w\in \tilde{F}$}
$$
and $h_{\tilde{F}}(F) \ge \sup\{h_{\nu}(F) :  \lambda^{c}(\nu) \in [a,b ]\}-8 \gamma$.
In particular $F$ is contained in the irregular set $\tilde{L}_{\psi}$. Furthermore, since $\lambda^{c}(\nu_{i}) \in (a , b)$, for $i = 1, 2$, we have that $\tilde{F} \subset E_{a,b}$. Thus 
$$\sup\{h_{\nu}(F) :  \lambda^{c}(\nu) \in [a,b ]\}-8 \gamma \leq h_{\tilde{F}}(F) \leq h_{E_{a,b}}(F) \leq \sup\{h_{\nu}(F) :  \lambda^{c}(\nu) \in [a,b ]\} \Rightarrow 
$$
$$
h_{E_{a,b}}(F) = \sup\{h_{\nu}(F) :  \lambda^{c}(\nu) \in [a,b ]\}.$$
 \end{proof}

 \begin{proposition}
 i) If $[a , b] \subset [0 , \lambda^{c}_{\min}]$ then $h_{L^{c}_{a , b}}(F) = bt_{0} + h_{top}(g)$;\\
 
 ii) If $[a , b] \subset [\lambda^{c}_{\min} , \lambda^{c}_{\max}]$ then $h_{L^{c}_{a , b}}(F) = h_{top}(F) - \inf_{s \in [a , b]}I(s) = h_{top}(F) - I(d)$; where $d = \lambda^{c}(\mu_{0})$ when $\lambda^{c}(\mu_{0}) \in [a,b]$, $d = b$ when $b \leq \lambda^{c}(\mu_{0})$ and $d = a$ when $\lambda^{c}(\mu_{0}) \leq a$. 
 \end{proposition}
 \begin{proof}
 i)] Let $d \in [0 , \lambda^{c}_{\min}]$ be and let $\nu$ be an $F-$invariant probability such that $\lambda^{c}(\nu) = d$. 
 Then $h_{\nu}(F) - t_{0}\lambda^{c}(\nu) \leq P(t_{0}) = h_{top}(g)  \Rightarrow h_{\nu}(F) \leq h_{top}(g) + dt_{0}$.
 
 On the other hand, let $\eta $ be the probability that appears in the item (3) of the definition of $F$ and $\mu_{t_{0}}$ an acumulattion point of $\mu_{t\phi^{c}}$ when $t$ converges to $t_{0}$ . Then $\lambda^{c}(Leb \times \eta) = 0$ and $h_{Leb \times\eta}(F) = h_{top}(g)$. Define $\mu_{d} := \frac{d}{\lambda^{c}_{\min}}\mu_{t_{0}} + (1-\frac{d}{\lambda^{c}_{\min}})Leb \times \eta$. Thus $\lambda^{c}(\mu_{d}) = d $ and $h_{\mu_{d}}(F) = h_{top}(g) +dt_{0}. $
 
 Therefore, applying the previous lemma, we conclude that $h_{L^{c}_{a , b}}(F) = bt_{0} + h_{top}(g)$ for all $[a , b] \subset [0 , \lambda^{c}_{\min}].$
 
 \
 
 ii)] Suppose that $[a , b] \subset [\lambda^{c}_{\min} , \lambda^{c}_{\max}]$. Applying the Lemma \ref{lemX} and Remark \ref{rem1} we conclude  that $h_{L^{c}_{a , b}}(F) = h_{top}(F) - \inf_{s \in [a , b]}I(s)$. Thus, follows the item ii) of the Lemma \ref{lemmaI} that $\inf_{s \in [a , b]}I(s) = I(d)$; where $d = \lambda^{c}(\mu_{0})$ when $\lambda^{c}(\mu_{0}) \in [a,b]$, $d = b$ when $b \leq \lambda^{c}(\mu_{0})$ and $d = a$ when $\lambda^{c}(\mu_{0}) \leq a$.
 \end{proof}
 
 It follows from the previous result that $h_{L^{c}_{0,\lambda^{c}_{\min}}}(F) = \lambda^{c}_{\min}t_{0} + h_{top}(g)$. On the other hand,
 $\lim_{s \mapsto \lambda^{c}_{\min}} I(s) = \lim_{t \mapsto -t_{0}}t\mathcal{E}'(t) - \mathcal{E}(t) = -\lambda^{c}_{\min}t_{0} - h_{top}(g) + h_{top}(F)$. Thus, if $[a , b] \cap [\lambda^{c}_{\min} , \lambda^{c}_{\max}]$ then $h_{L^{c}_{a,b}}(F)  = h_{L^{c}_{\lambda^{c}_{\min},b}}(F)$. Therefore, using the previous proposition and the properties of Lemma \ref{lemmaI}, we 
 complete the proof of the Theorem \ref{mainthC}.

\subsubsection{Dimension spectrum for Lyapunov exponents}

In this section $f: \Sc^{1} \rightarrow \Sc^{1}$ will be a transitive non invertible $C^{1}-$local diffeomorphism with $Df$ Holder continuous.

As we mentioned previously, the respective version of the Theorem \ref{mainthC} also is valid for $f$. So it remains for us to show the item (ii) of Corollary \ref{mainthD}.

For the Hausdorff dimension spectrum,  we need to recall the results of Hofbauer \cite{Ho10} in our context. For $u, v \in \mathbb{R}$ with $u \leq v$, Hofbauer considers the set

$$
M_{u, v}=\{x \in A: u \leq \underline{\lambda}(x) \leq \bar{\lambda}(x) \leq v\}
$$
where $\overline{\lambda}(x)=\limsup \frac{1}{n}S_n \log|Df(x)|$ and $\underline{\lambda}(x)=\liminf \frac{1}{n}S_n\log|Df(x)|$ are, respectively, the upper and lower Lyapunov exponents associated with the point $x$. 

Hofbauer introduces the pressure function $\tau(s):=P(-s)$, which  is the flipped version of our pressure function, and its Legendre transform $\hat{\tau}$ defined by \begin{equation}\label{LegTransf}
     \hat{\tau}(a)=\inf _{s \in \mathbb{R}}(\tau(s)-a s)
 \end{equation} on the set $H := \{a \in \R : as \leq \tau(s) \text{ for all } s \in \R\}$. Since $\tau$ is convex, $H$ is a nonempty interval.
In his paper, Hofbauer uses the Legendre transform $\hat{\tau}$ to characterize the entropy spectrum of Birkhoff averages. In fact, it has been proved that 
$$h_{M_{u, v}}(f) =\max _{a \in H \cap[u, v]} \hat{\tau}(a).$$

In this paper, it is also introduced the function $\check\tau$ defined as the following

$$
\check{\tau}(a)=\frac{1}{a} \hat{\tau}(a)=\frac{1}{a} \inf _{s \in \mathbb{R}}(\tau(s)-a s) \quad \text { for all } a \in H \backslash\{0\}
$$
which is used to characterise the dimension spectrum of Lyapunov exponents. In fact, it has been proved that 

$$HD(M_{u, v})=\max _{a \in[u, v]} \check{\tau}(a) = \max _{a \in[u, v]}\frac{1}{a} \hat{\tau}(a).$$

We will now relate to our previous context. Note initially that $H = [0 , \lambda_{\max}] = L(f)$ and 
$$\hat{\tau}(a) = h_{M_{a,a}}(f) = h_{L_{a,a}}(f).$$ In particular, if $a \in [0 , \lambda_{\max}]$ then $\hat{\tau}(a) = a t_{0}$; and if $a \in [\lambda_{\max} , \lambda_{\min}]$ then $\hat{\tau}(a) = I(a)$. Moreover:
$$L_{a , b} \subset M_{a,b} \Rightarrow HD(L_{a , b}) \leq HD(M_{a , b}) = \max_{c \in[a, b]}\frac{1}{c} \hat{\tau}(c) = \frac{1}{c} \hat{\tau}(c) =$$
$$HD(M_{c, c}) = HD(L_{c, c})
\Rightarrow HD(L_{a , b}) = HD(M_{a , b}).$$
Similarly, $HD(E_{a , b}) = HD(M_{a , b})$. In order
to finish the proof of the Corollary \ref{mainthD} it is enough to describe the properties of the function $\hat{\tau}.$

\begin{proposition}\label{propct}
$\check\tau$ is strictly concave, strictly decreasing and analytic on $(0 , \lambda_{\max})$. If $\lambda_{\min}>0$ then $\check\tau(a)\equiv t_0$ for $a \in (0,\lambda_{\min}]$ and is not analytic in $\lambda_{\min}$. 
\end{proposition}
 
\begin{proof}
Given $a \in (\lambda_{\min},\lambda_{\max})$ we have  $\check\tau(a) = \frac{h_{top}(F) - I(a)}{a}$.  Since $\check\tau(a)=\dfrac{\hat\tau(a)}{a}$, then $\check\tau_{|(\lambda_{\min},\lambda_{\max})}$  is analytic. Now derivating for $a$:
$$ \check\tau'(a)=\dfrac{\hat\tau'(a)\cdot a - \hat\tau(a)}{a^2}=\dfrac{\hat\tau'(a)-\hat\tau(a)/a}{a}.$$
Where the numerator on the last fraction is negative, since $\hat\tau$ is a strictly concave function. Derivating one more time

   $$\check\tau''(a)=\dfrac{[\hat\tau''(a) \cdot a+\hat\tau'(a)-\hat\tau'(a)]\cdot a^2 + [\hat\tau'(a)\cdot a -\hat\tau(a)]\cdot 2a}{a^4}$$
   $$=\dfrac{\hat\tau''(a)}{a}+\dfrac{2[\hat\tau'(a)\cdot a -\hat\tau(a)]}{a^3}$$
 Where the numerator of each fraction is negative, again because $\hat\tau$ is strictly concave.
 Therefore $\check\tau$ is strictly concave and decreasing in $(\chi_{min},\chi_{max}).$
 
 Now suppose $\chi_{min}>0$, then for any $a \in (0, \chi_{min}]$ we have 
 $$\check\tau(a)=\dfrac{\hat\tau(a)}{a}=\dfrac{a\cdot t_0}{a}=t_0.$$
Finally, $\check\tau$ is not analytic on $\chi_{min}$, otherwise it should be constant on a neighbourhood of $\chi_{min}$, and we conclude the argument.
\end{proof}

The expected shapes for $\check{\tau}$ are given by the figures below:

\;

\begin{figure}[h]%
    \centering
{\includegraphics[width=5cm]{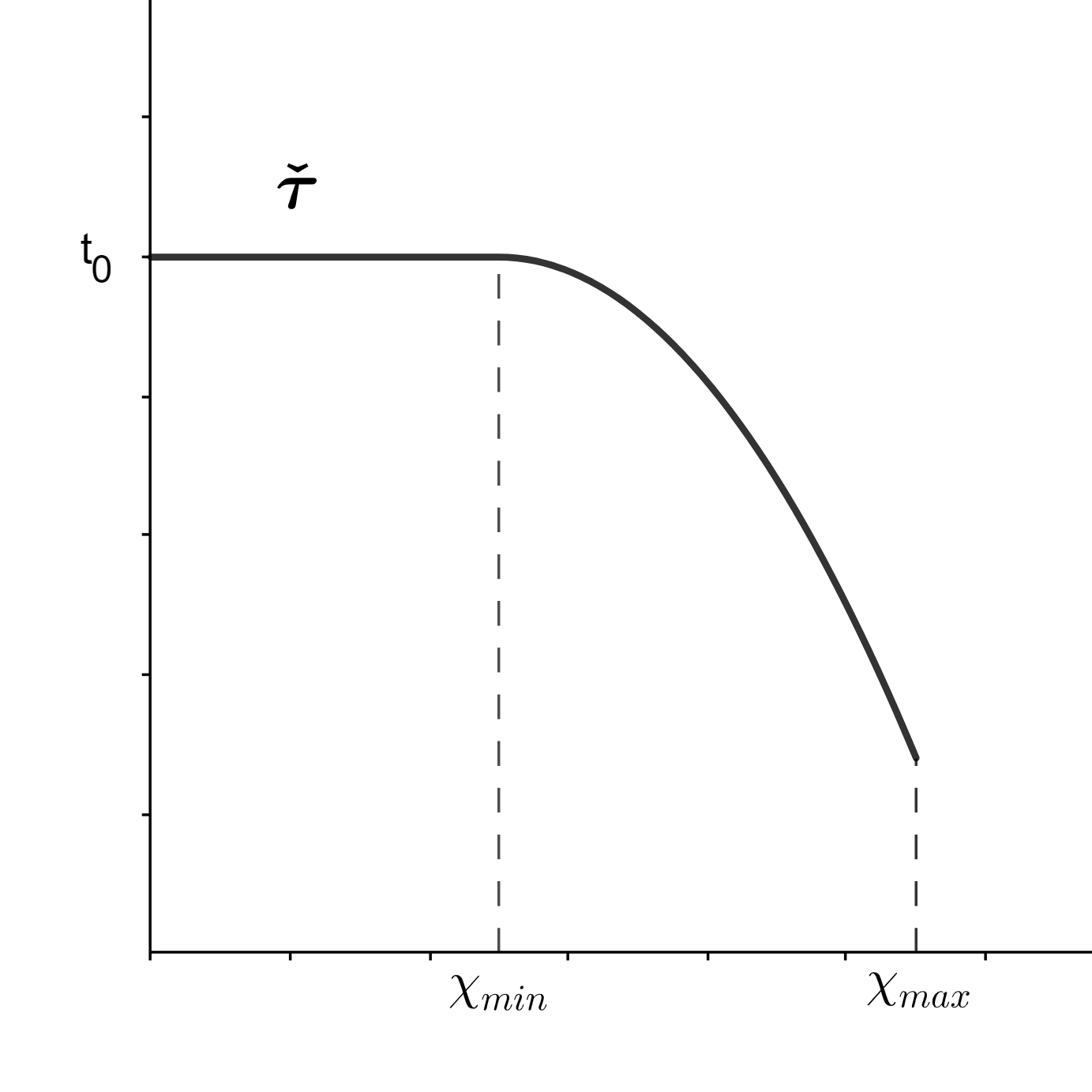} }%
    \qquad
{\includegraphics[width=5cm]{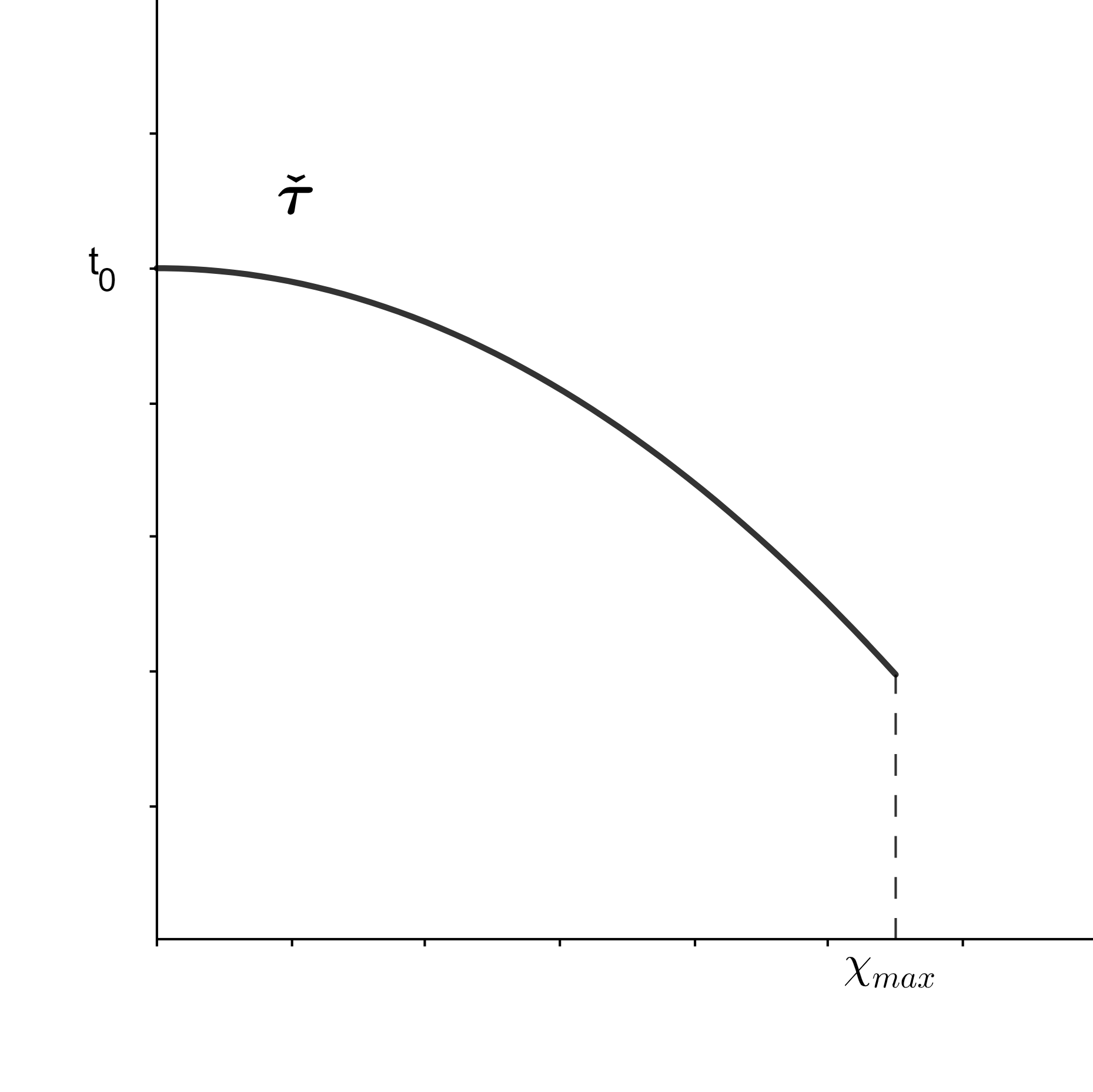} }%
    \caption{Expected shapes for $\check\tau(a)=\dim_H(L_a)$}%
    \label{fig:example}%
    
\end{figure}

\section{Differentiability on transition parameter}\label{difparam}

Another interesting question, regarding the parameter, is under what conditions the pressure function is differentiable or not at $t_0$. 
We will focus our discussion on the context of the circle, that is, Corollary \ref{mainthD}.

Note that if $f$ admits an absolutely continuous invariant probability with positive Lyapunov exponent $\chi_\mu>0$, then $t_0=1$ and the left derivative is at most $-\chi_\mu < 0$, thus the pressure function is not differentiable on $t_0=1$ (see \cite{BC21}). On a large class of maps this is well understood: interval maps which are expanding except on a finite number of indifferent fixed points.  


 In the work of Pianigiani, \cite{Pi80}, the maps can be modelled as local diffeomorphism. Indeed consider transformations $T$ for which $\left|T^{\prime}(x)\right|=1$ on a finite number of fixed points $x_{1}, \cdots, x_{n}$. Take a neighborhood $U_{i}$ of $x_{i}$ meaning right, left or bilateral neighborhood if

$$
T(x) \rightarrow x_{i} \quad \text { and } \quad\left|T^{\prime}(x)\right| \rightarrow 1
$$
as $x \rightarrow x_{i}^+, x \rightarrow x_{i}^-$ or $x \rightarrow x_{i}$

A fixed point $x_{i}$ with $\left|T^{\prime}\left(x_{i}\right)\right|=1$ is said to be \textit{regular} if there exists a neighborhood $U_{i}$ of $x_{i}$ and a real $\alpha, 0<\alpha<1$ such that

\begin{align}\label{regular}
  \left|T^{\prime \prime}(x)\right| \geqq L\left|x-x_{i}\right|^{-\alpha}, \quad x \in U_{i}  
\end{align}

Note that this assumes that $T'$ and thus $(T^{-1})'$ are differentiable for $x\neq x_i$ in $U_i$, but not in $x_i$ as the right hand side of \ref{regular} tends to infinity, as $x$ approaches $x_i$. 

The following theorem implies 
\begin{theorem} {\bf (Pianigiani)}\label{Pianigiani}
Let $T$ be a countable piecewise $C^{1}$ transformation with finite image. Let $U_{1}, \cdots, U_{n}$ be neighborhoods of $x_{1}, \cdots, x_{n}$ such that $\left|T^{\prime}(x)\right|$ decreases monotonically to 1 as $x \in U_{i}$ goes to $x_{i}$. If
$$
\left|T^{\prime}(x)\right| \geqq \lambda>1, \quad x \in[0,1] \backslash \bigcup_{i=1}^{n} U_{i}
$$
then there exists an absolutely continuous $\sigma$-finite measure $\mu$ invariant under $T$.

Moreover, if all the $x_{i}$ are regular then $\mu$ is finite.
\end{theorem}

For our context, a $C^{1+p}$-local diffeomorphism of the circle with regular indifferent fixed points, we can translate this regularity condition into both $f$ and its inverse branches having bi-H\"older derivative, and possibly, but not necessarily, piecewise $C^2$ except on the $x_i$.
So we could expect lower regularity of the function outside the indifferent fixed points, however in most examples the map is in fact $C^2$ except on these fixed points.

 Finally, we get an absolutely continuous invariant probability for a large class of intermittent maps on the circle, such as the Manneville-Pomeau for $0<p<1$

\begin{equation}
 f_p(x)=\begin{cases}
 x+2^p x^{1+p}, \text{ if } 0\leq x \leq 1/2 \\
 x-2^{p}(1-x)^{1+p}, \text{ if } 1/2 < x \leq 1.
 \end{cases}
 \end{equation}

\noindent In fact for these maps we even get the equality on (\ref{regular}) for $x_1=0\sim 1$, $\alpha=(1-p)$ and $L=2^p(1+p)p $. Thus $0 \sim 1$ is a regular fixed point.


 The author gives conditions for this a.c.i.p.  to be unique, which are not easy to check in a more general setting. However, this a.c.i.p.  is obtained via an inducing scheme on intervals outside the indifferent points, that is via an expanding first return map. Since our maps are topologically exact, this first return map is transitive expanding, thus it admits only one a.c.i.p., making the a.c.i.p.  for our maps $f$ also unique.

Now let $\mu$ be such an a.c.i.p.  for an intermittent map of the circle with regular fixed points. Since $|f'|> 1$ except on a finite number of points and $\mu$ is absolutely continuous with respect to Lebesgue with density positive on an open set , we have that $\chi_\mu=\int \log|Df| d\mu > 0$. Therefore by uniqueness of the a.c.i.p., we have the equivalences:
\begin{enumerate}
    \item the indifferent fixed points are regular, that is, having a strictly $C^{1+p}$ expansion, 
    \item the map $f$ admits an unique finite a.c.i.p.  
    \item the pressure function is not differentiable at $t_0=1$; 
\end{enumerate}

\


\vspace{.3cm}
\subsection*{Acknowledgements.}
This work is part of the second and third author's PhD thesis at the Federal University of Bahia. TB was partially supported by Conselho Nacional de Desenvolvimento Científico e Tecnologico - CNPq (Grants PQ-2021 and Universal-18/2021). VC and AF were supported by CAPES-Brazil. 

\bibliographystyle{alpha}

\end{document}